\documentclass[preprint,12pt,numbers]{elsarticle}

\usepackage{amsmath,amsfonts,amssymb,amsthm}
\usepackage{mathtools}
\usepackage{graphicx}
\usepackage{subcaption}
\usepackage{booktabs}
\usepackage{multirow}
\usepackage[ruled,vlined]{algorithm2e}
\usepackage{wrapfig}
\usepackage[final]{microtype}
\usepackage{enumitem}
\usepackage{comment}
\usepackage{hyperref}
\hypersetup{hypertexnames=false}
\usepackage{xcolor}
\usepackage{placeins}

\allowdisplaybreaks
\theoremstyle{plain}
\theoremstyle{definition}
\newtheorem{definition}{Definition}[section]

\newtheorem{lemma}{Lemma}[section]
\newtheorem{theorem}{Theorem}[section]

\newtheorem{corollary}{Corollary}[section]
\theoremstyle{remark}
\newtheorem{remark}{Remark}[section]

\journal{Computers and Mathematics with Applications}

\begin{document}
	
	\begin{frontmatter}
		
		
		\title{SHANG++: Robust Stochastic Acceleration under Multiplicative Noise}
		
		\author[aff1]{Yaxin Yu}
		\ead{yaxin_yu@stu.scu.edu.cn}
		
		\author[aff2]{Long Chen\corref{cor1}}
		\ead{lchen7@uci.edu}
		
		\author[aff1]{Minfu Feng}
		\ead{fmf@scu.edu.cn}
		
		\affiliation[aff1]{organization={School of Mathematics, Sichuan University},
			city={Chengdu},
			state={Sichuan},
			postcode={610065},
			country={China}}
		
		\affiliation[aff2]{organization={Department of Mathematics, University of California, Irvine},
			city={Irvine},
			state={CA},
			postcode={92697},
			country={USA}}
		
		\cortext[cor1]{Corresponding author.}
		
		\begin{abstract}
			Under the multiplicative noise scaling (MNS) condition, original Nesterov acceleration is provably sensitive to noise and may diverge when gradient noise overwhelms the signal. In this paper, we develop two accelerated stochastic gradient descent methods by discretizing the Hessian-driven Nesterov accelerated gradient flow. We first derive SHANG, a direct semi-implicit discretization that already improves stability under MNS. We then introduce SHANG++, which adds a damping correction and achieves faster convergence with greater noise robustness. We establish convergence guarantees for both convex and strongly convex objectives under MNS, together with explicit parameter choices. In our experiments, SHANG++ performs consistently well across convex problems and applications in deep learning. In a dedicated noise experiment on ResNet-34, a single hyperparameter configuration maintains accuracy within one percentage point of the noise-free setting. Across all experiments, SHANG++ outperforms existing accelerated methods in robustness and efficiency, with minimal parameter sensitivity.
		\end{abstract}
		
		
		
		\begin{keyword}
			Stochastic gradient descent \sep Momentum method  \sep Acceleration \sep Multiplicative noise scaling \sep Convergence analysis
			
			
			
		\end{keyword}
		
	\end{frontmatter}
	
		
		
		\section{Introduction}
		Empirical Risk Minimization (ERM) is central to modern large-scale machine learning, including deep neural networks and reinforcement learning~\cite{HastieTibshiraniFriedman2009}. Given a large dataset $\{(X_i,Y_i)\}_{i=1}^{N}$, where $Y_i$ denotes the label of data $X_i$ and $N \gg 1$, the training objective is
		\begin{equation}\label{problem}
			\min_{x} f(x), \qquad 
			f(x)=\frac{1}{N}\sum_{i=1}^{N} f_i(x),
		\end{equation}
		where $x$ denotes the network parameters and $f_i(x)$ is the loss associated with sample $(X_i,Y_i)$. We use $x$ instead of $\theta$ to keep the notation consistent with the optimization formulation. Efficiently computing a minimizer $x^{\star} = \arg\min_{x} f(x)$ is critical for training neural networks on large datasets.
		
		Exact gradient evaluation is expensive when $N$ is large. Stochastic gradient descent (SGD) uses an unbiased stochastic gradient estimator. At iteration $k$, we draw a mini-batch $B_k\subset\{1,\dots,N\}$ of size $M$ uniformly at random, with or without replacement, independently of the past, and form
			\begin{equation}\label{equsg}
				g(x_k) = \frac{1}{M}\sum_{i\in B_k} \nabla f_i(x_k).
			\end{equation}
			Let \(\{\mathcal F_k\}_{k\ge0}\) be the natural filtration representing the information available before sampling the mini-batch at \(x_k\). Then \(x_k\) is \(\mathcal F_k\)-measurable, while \(B_k\) is sampled independently of \(\mathcal F_k\). Hence the mini-batch estimator in~\eqref{equsg} is conditionally unbiased:
			\begin{equation}\label{eq:unbiase}
				\mathbb{E}\left[g(x_k)\mid \mathcal{F}_k\right]=\nabla f(x_k).
			\end{equation}
		
		The estimator $g(x_k)$ reduces the cost of computing $\nabla f(x_k)$ but introduces noise. In regimes such as small-batch training or over-parameterized models, the variance can scale with, and may dominate, the signal $\|\nabla f(x_k)\|^2$. This effect is modeled by the multiplicative-noise scaling (MNS) condition~\cite{WuDuWard2019,WuWangSu2022a,GuptaSiegelWojtowytsch2024}. The work in \cite{Hodgkinson2021} shows that multiplicative noise can change the geometry of the loss landscape, beyond the smoothing effects of additive noise.
		
		\begin{definition}[Multiplicative Noise Scaling (MNS)]\label{MNS}
			The stochastic gradient estimator satisfies the MNS condition if there exists $\sigma \ge 0$ such that
			\begin{equation}\label{eq:mns}
				\mathbb{E}\left[\| g(x_k) - \nabla f(x_k) \|^2 \mid \mathcal{F}_k \right] \le \sigma^2 \| \nabla f(x_k)\|^2, \quad \forall k\geq 0.
			\end{equation}
		\end{definition}
		The MNS condition is a noise model, not a consequence of unbiasedness. It adds a relative-noise structure to the stochastic gradient estimator.
		
		SGD slows down when the condition number of $f$ is large. Momentum methods, such as Heavy-Ball (HB)~\cite{Polyak1964} and Nesterov accelerated gradient (NAG)~\cite{Nesterov1983}, are widely used to speed up convergence. Adam (Adaptive Moment Estimation)~\cite{KingmaBa2015} is widely used in training deep neural networks. It combines momentum with adaptive step sizes to improve stability and convergence. In this work, we focus on acceleration rather than adaptivity.
		
		Accelerated variants of SGD have been extensively studied. However, momentum methods are highly sensitive to stochastic noise~\cite{DevolderGlineurNesterov2014,AujolDossal2015,LiuEtAl2018}, and stability depends critically on parameter choices~\cite{KidambiEtAl2018,LiuBelkin2020, AssranRabbat2020,GaneshEtAl2023}. Gupta et al.~\cite{GuptaSiegelWojtowytsch2024} showed that, under MNS with $\sigma \ge 1$, NAG fails to converge even for convex and strongly convex objectives. In practice, the benefits of momentum often depend on large mini-batches, which reduce variance and make the stochastic dynamics closer to the deterministic one.
		
		Several corrections have been developed to address this instability. Following \cite{jain2018}, many accelerated stochastic algorithms have been proposed~\cite{LiuBelkin2020,VaswaniBachSchmidt2019,even2021,Bollapragada2022,laborde2020,GuptaSiegelWojtowytsch2024,HermantEtAl2025}. Their goal is to retain acceleration while improving robustness to noise. Vaswani et al.~\cite{VaswaniBachSchmidt2019} introduced a four-parameter NAG variant with optimal accelerated rates. Liu and Belkin~\cite{LiuBelkin2020} proposed the Mass method with a three-parameter correction, though acceleration was proved only for over-parameterized linear models. Gupta et al.~\cite{GuptaSiegelWojtowytsch2024} developed AGNES with guarantees matching~\cite{VaswaniBachSchmidt2019}. Hermant et al.~\cite{HermantEtAl2025} analyzed SNAG, a four-parameter variant in Nesterov's framework~\cite{Nesterov2013}, and proved similar rates under mild tuning.
		
		These accelerated SGDs are theoretically strong for convex optimization but not robust to noise in deep learning tasks. In the reported tests, AGNES and SNAG can lose acceleration under high noise and may perform worse than SGD with recommended hyperparameters. For example, on CIFAR-100 with ResNet-50 and batch size $50$, SGD attains $58.31\%$ test accuracy whereas AGNES reaches $42.82\%$; see Section~\ref{test}. With smaller batches, both AGNES and SNAG show strong oscillations and require additional tuning. These observations motivate a stable, easy-to-tune method.
		
		Our goal is to develop a complementary stochastic acceleration framework that (i) retains accelerated convergence guarantees for convex problems, (ii) gives explicit noise-dependent parameter choices, and (iii) improves empirical stability with modest tuning. Our contributions are as follows.
		
		\begin{enumerate}
			\item We begin with SHANG, a stochastic extension of the HNAG method~\cite{ChenLuo2021}. Compared with classical NAG, the continuous-time dynamic underlying HNAG contains the Hessian-driven term $\nabla^{2}f(x)x'$, which captures curvature-dependent damping and gives a more refined continuous-time model of NAG. SHANG inherits this structure and leads to a Lyapunov analysis under MNS.
			
			\item We refine SHANG into SHANG++ by adding a correction term $-m(x_{k+1}-x_k)$ to the $x$-update. This term relaxes the coupled stepsize scaling in SHANG and offsets the MNS-induced changes of the effective constants $\mu$ and $L$. As a result, SHANG++ allows more flexible parameter choices and gives a sharper convergence factor under strong convexity.
			
			\item We prove convergence estimates for SHANG and SHANG++ for convex and strongly convex objectives under unbiasedness and MNS conditions. The rates show how $\sigma$ changes the admissible stepsize and the acceleration factor. The analysis is for the two-gradient recursions in Section~\ref{methods}. In deep-learning experiments, we use a one-gradient, index-shifted implementation with the same gradient cost as standard stochastic momentum methods; its use for nonconvex training is empirical.
			
			\item We evaluate SHANG and SHANG++ on convex optimization, image classification, and image reconstruction tasks, including MNIST, CIFAR-10, and CIFAR-100. In the reported experiments, SHANG++ is more stable than NAG, AGNES, and SNAG under high multiplicative noise, and is competitive with Adam.
			
			\item We further examine robustness to multiplicative noise. For the tested noise levels $\sigma \le 0.5$, SHANG++ stays close to its noise-free accuracy, while maintaining accuracy within one percentage point of the noise-free setting in the main ResNet-34 experiment. These results support the practical robustness of the method, although the current theory is restricted to convex objectives.
			
		\end{enumerate}
		
		The acronym SHANG stands for Stochastic Hessian-driven Accelerated Nesterov Gradient. We choose it as a pronounceable name with a positive connotation; in Chinese pinyin, \emph{shang} means upward or top.
		
		\paragraph{Notation}
		Throughout this paper, we denote by $\langle \cdot, \cdot \rangle$ the standard Euclidean inner product in $\mathbb{R}^d$, and by $\| \cdot \|$ its induced norm. Let $x^{\star}=\arg\min_{x}f(x)$ denote a global minimizer of $f$.
		
		Let $f: \mathbb{R}^d \to \mathbb{R}$ be differentiable. The Bregman divergence of $f$ between $x, y \in \mathbb{R}^d$ is 
		\begin{equation}\label{bregman}
			D_f(y, x) := f(y) - f(x) - \langle \nabla f(x), y - x \rangle.
		\end{equation}
		The function $f$ is convex if $D_f(y,x)\ge 0$ for all $x,y\in\mathbb{R}^d$. Moreover, $f$ is $\mu$-strongly convex with parameter $\mu\ge 0$ if $D_f(y,x)\ge \frac{\mu}{2}\|y-x\|^2$.
		Function $f$ is $L$-smooth, for some $L > 0$, if its gradient is $L$-Lipschitz: 
		\begin{equation}\label{smooth}
			\|\nabla f(y) - \nabla f(x)\| \leq L \|y - x\|, \quad \forall x, y \in \mathbb{R}^d.
		\end{equation}
		
		Let $\mathcal{S}_{\mu, L}$ be the class of all differentiable functions that are both $\mu$-strongly convex and $L$-smooth on $\mathbb{R}^d$. For $f \in \mathcal{S}_{\mu, L}$, the Bregman divergence satisfies
		\begin{equation}\label{bregman2}
			\frac{\mu}{2} \| x - y \|^2 \leq D_f(x, y) \leq \frac{L}{2} \| x - y \|^2, \quad \forall x, y \in \mathbb{R}^d.
		\end{equation}
		Bregman divergence here is used as an analytical tool in the Lyapunov analysis. Parameters $\mu$ and $L$ are treated as known hyperparameters for the given problem. Their adaptivity is beyond the scope of this work.
		
		The remainder of the paper is organized as follows. We present SHANG and SHANG++ in Section~\ref{methods}. Experimental results are reported in Section~\ref{test}. The complete convergence analysis is given in Section~\ref{analysis}, and Section~\ref{conclusion} concludes the paper.
		
		\section{Stochastic Hessian-driven Accelerated Nesterov Gradient}\label{methods}
		
		\subsection{Hessian-driven Nesterov Accelerated Gradient Flow}
		To accelerate gradient descent, Polyak introduced momentum, modeled by the Heavy-Ball (HB) flow~\cite{Polyak1964}:
		\begin{equation}\label{eq:HB}
			x^{\prime \prime} + \theta x^{\prime} + \eta \nabla f(x) = 0.
		\end{equation}
		However, the discrete HB method $x_{k+1} = x_k - \gamma \nabla f(x_k) + \beta (x_k - x_{k-1})$ may not converge; see \cite{LessardRechtPackard2016,GoujauTaylorDieule2025Provable} for examples.
		
		We build on the second-order dynamical system introduced in \cite{ChenLuo2019optimization,ChenLuo2021}, known as the Hessian-driven Nesterov Accelerated Gradient (HNAG) flow:
		\begin{equation}\label{equ1}
			\gamma x'' + (\gamma + \mu)x' + \beta \gamma \nabla^2 f(x) x' + (1+\mu \beta) \nabla f(x) = 0,
		\end{equation}
		where $\beta>0$ is a parameter and $\gamma$ is a time-scaling function. Compared with the classical HB flow~\eqref{eq:HB}, the additional Hessian-driven term $\nabla^2 f(x)x'$ captures how the local geometry of $f$ affects the damping strength of the dynamics. As shown in \cite{ChenLuo2019optimization}, this curvature-aware mechanism gives a more accurate continuous-time description of NAG.
		
		The second-order ODE~\eqref{equ1} can be written as the first-order system
		\begin{equation}\label{equ2}
			\begin{aligned}
				x' &= v - x - \beta \nabla f(x), \quad &
				v' &= \frac{\mu}{\gamma}(x - v) - \frac{1}{\gamma} \nabla f(x), \quad &
				\gamma' &= \mu - \gamma,
			\end{aligned}
		\end{equation}
		which removes the explicit dependence on $\nabla^2 f(x)$.
		
		\subsection{Stochastic Hessian-driven Accelerated Nesterov Gradient method}\label{sec:shang}
		Applying a semi-implicit discretization to the first-order HNAG system~\eqref{equ2} and replacing deterministic gradients by unbiased stochastic estimates, we obtain the Stochastic Hessian-driven Accelerated Nesterov Gradient method, abbreviated as SHANG:
		\begin{subequations}\label{SHANG}
			\begin{align}
				\frac{x_{k+1}-x_k}{\alpha_k}
				&= v_k - x_{k+1} - \beta_k g(x_k) \label{SHANG:a}\\
				\frac{v_{k+1}-v_k}{\alpha_k}
				&= \frac{\mu}{\gamma_k}(x_{k+1}-v_{k+1}) - \frac{1}{\gamma_k}g(x_{k+1}) \label{SHANG:b}\\
				\frac{\gamma_{k+1}-\gamma_k}{\alpha_k}
				&\le \mu - \gamma_{k+1} \label{SHANG:c}
			\end{align}
		\end{subequations}
		with given $x_0$, $v_0$, and $\gamma_0$. For the time-scaling variable, we use the one-sided schedule condition~\eqref{SHANG:c}, which allows more choices of $\{\gamma_k\}$.
		
		In the strongly convex case, we fix $\gamma_k\equiv \mu$ and take a constant step size $\alpha_k\equiv \alpha$. In the convex case, we set $\mu=0$ and allow both $\alpha_k$ and $\gamma_k$ to vary; then \eqref{SHANG:c} implies that $\{\gamma_k\}$ is decreasing. The coupling $\beta_k>0$ depends on $(\alpha_k,\gamma_k)$ and typically scales as $(1+\sigma^2)\alpha_k/\gamma_k$. Thus SHANG is a one-parameter scheme in the strongly convex case and a two-parameter scheme in the convex case.
		
		For the analysis, we introduce the auxiliary variable
		\begin{equation}\label{auxi1}
			x^+ := x - \alpha \beta g(x),
		\end{equation}
		which is one stochastic gradient step from \(x\) with step size \(\alpha \beta>0\).
		With this notation, SHANG~\eqref{SHANG} admits the following augmented but equivalent form:
		\begin{subequations}\label{SHANG2}
			\begin{align}
				\frac{x_{k+1}-x_k^+}{\alpha_k}
				&= v_k - x_{k+1}, \label{SHANG2:a}\\
				\frac{v_{k+1}-v_k}{\alpha_k}
				&= \frac{\mu}{\gamma_k}(x_{k+1}-v_{k+1}) - \frac{1}{\gamma_k}g(x_{k+1}), \label{SHANG2:b}\\
				\frac{\gamma_{k+1}-\gamma_k}{\alpha_k}
				&\le \mu - \gamma_{k+1}, \label{SHANG2:c}\\
				x_{k+1}^+
				&= x_{k+1}-\alpha_{k+1}\beta_{k+1}g(x_{k+1}). \label{SHANG2:d}
			\end{align}
		\end{subequations}
		Here the last equation only defines the next auxiliary point and does not introduce an additional algorithmic update. The auxiliary update~\eqref{SHANG2:d} gives the following descent estimate. In this one-step formulation, \(x_{k+1}\) is fixed before the fresh stochastic gradient \(g(x_{k+1})\) is sampled; the corresponding pre-sampling information is denoted by \(\mathcal F_{k+1}\). The update \(x_{k+1}^+\) is stochastic.
		
		\begin{lemma}\label{lemma2}
			Suppose that $f$ is $L$-smooth and that $x_{k+1}$ is $\mathcal F_{k+1}$-measurable. Conditional on $\mathcal F_{k+1}$, let $g(x_{k+1})$ be an unbiased stochastic gradient estimator of $\nabla f(x_{k+1})$ satisfying the MNS condition~\eqref{eq:mns}.
			Let $x_{k+1}^+$ be obtained by \eqref{SHANG2:d} with stepsize 
			$$
			0<\alpha_{k+1}\beta_{k+1}\le \frac{1}{L(1+\sigma^2)}.
			$$
			Then
			$$
			\mathbb E\left[f(x_{k+1}^+)\mid\mathcal F_{k+1}\right]
			\le
			f(x_{k+1})-\frac{\alpha_{k+1}\beta_{k+1}}{2}\|\nabla f(x_{k+1})\|^2.
			$$
		\end{lemma}
		
		\begin{proof}
			By $L$-smoothness,
			$$
			\begin{aligned}
				f(x_{k+1}^+)
				&\le
				f(x_{k+1})
				+\left\langle \nabla f(x_{k+1}),x_{k+1}^+-x_{k+1}\right\rangle
				+\frac{L}{2}\|x_{k+1}^+-x_{k+1}\|^2 \\
				&=
				f(x_{k+1})
				-\alpha_{k+1}\beta_{k+1}\left\langle \nabla f(x_{k+1}),g(x_{k+1})\right\rangle
				+\frac{L\alpha_{k+1}^2\beta_{k+1}^2}{2}\|g(x_{k+1})\|^2 .
			\end{aligned}
			$$
			Taking conditional expectation with respect to $\mathcal F_{k+1}$, and using conditional unbiasedness, gives
			$$
			\mathbb E\left[\left\langle \nabla f(x_{k+1}),g(x_{k+1})\right\rangle\mid\mathcal F_{k+1}\right]
			=
			\|\nabla f(x_{k+1})\|^2.
			$$
			Moreover, the MNS condition implies
			$$
			\mathbb E[\|g(x_{k+1})\|^2\mid\mathcal F_{k+1}]
			\le
			(1+\sigma^2)\|\nabla f(x_{k+1})\|^2.
			$$
			Therefore
			$$
			\begin{aligned}
				\mathbb E\left[f(x_{k+1}^+)\mid\mathcal F_{k+1}\right]
				&\le
				f(x_{k+1})
				-\alpha_{k+1}\beta_{k+1}\|\nabla f(x_{k+1})\|^2 \\
				&\quad
				+\frac{L\alpha_{k+1}^2\beta_{k+1}^2(1+\sigma^2)}{2}
				\|\nabla f(x_{k+1})\|^2  \\
				&\le
				f(x_{k+1})
				-\frac{\alpha_{k+1}\beta_{k+1}}{2}\|\nabla f(x_{k+1})\|^2,
			\end{aligned}
			$$
			where the last step follows from $0<\alpha_{k+1}\beta_{k+1}\le 1/(L(1+\sigma^2)).$
		\end{proof}
		
		\begin{remark}[Role of the auxiliary variable $x^+$]
				The auxiliary variable $x^+$ is used only in the convergence analysis. This reformulation avoids mixed stochastic terms at two consecutive iterates and reveals the telescoping structure in the Lyapunov estimate. In implementation, $x^+$ is not updated separately.
			\end{remark}
		
		Define the augmented variables $z_k := (x_k, v_k)$, $z_k^{+} := (x_k^{+}, v_k)$, $z^\star := (x^\star, x^\star)$, and the discrete Lyapunov function
		\begin{equation}\label{equ6}
			\mathcal{E}(z_k^+; \gamma_{k}) =  f(x_k^+) - f(x^{\star}) + \frac{\gamma_{k}}{2} \| v_k - x^{\star} \|^2.
		\end{equation}
		The following theorem gives decay of $\mathbb{E}\left[\mathcal{E} (z_{k+1}^+; \gamma_{k+1})\right]$.
		
		\begin{theorem}\label{theorem1}
			Let $f\in\mathcal{S}_{\mu,L}$. Starting from $x_0=v_0$, generate $(x_k,v_k)$ by~\eqref{SHANG}. Let the stochastic gradients be conditionally unbiased and satisfy the MNS condition~\eqref{eq:mns}. Define $x_k^+$ by~\eqref{auxi1}. Set $\beta_k = (1+\sigma^2)\alpha_k/\gamma_k$.
			\begin{itemize}
				\item[(1)] When $0 < \mu < L < \infty$, take $\alpha_k\equiv\alpha$ and $\gamma_k\equiv\mu$. If $0 < \alpha \le \frac{1}{1+\sigma^2}\sqrt{\frac{\mu}{L}}$, then for all $k \ge 0$,
				$$
				\mathbb{E}\left[f(x_{k+1}^+) - f(x^{\star}) + \frac{\mu}{2} \| v_{k+1} - x^{\star} \|^2 \right] \le (1+\alpha)^{-(k+1)} 	\mathbb{E}[\mathcal{E}(z_0^+; \mu)].
				$$
				\item[(2)] When $\mu=0$, set $\alpha_k = \frac{2}{k+1}$ and $\gamma_{k} = \alpha_k^2(1+\sigma^2)^2L$. Then for all $k \ge 0$,
				$$
				\mathbb{E}\left[f(x_{k+1}^+ ) - f(x^{\star}) + \frac{\gamma_{k+1}}{2} \| v_{k+1} - x^{\star} \|^2 \right] \le \frac{2}{(k+2)(k+3)}\mathbb{E}[ \mathcal{E}(z_0^+; \gamma_{0})].
				$$
			\end{itemize}
		\end{theorem}
		
		We give a proof outline and defer the full proof to Section~\ref{shanganalysis}. 
		Under our filtration convention, \(\mathcal F_{k+1}\) contains the information up to the formation of \(g(x_k)\). Hence \(x_k^+\) and \(x_{k+1}\) are \(\mathcal F_{k+1}\)-measurable, while \(g(x_{k+1})\) is sampled afterward. Therefore the conditional unbiasedness and MNS assumptions apply to \(g(x_{k+1})\) conditionally on \(\mathcal F_{k+1}\).
		
		\begin{proof}[Proof outline]
			Insert the intermediate term $\mathcal{E}(z_{k+1};\gamma_{k+1})$ and write
			\begin{equation}\label{equ147}
				\mathbb{E}\left[\mathcal{E}(z_{k+1}^+;\gamma_{k+1})-\mathcal{E}(z_{k+1};\gamma_{k+1})\right] +\mathbb{E}\left[\mathcal{E}(z_{k+1};\gamma_{k+1}) -\mathcal{E}(z_k^+;\gamma_k)\right].
			\end{equation}
			By Lemma~\ref{lemma2}, with $\alpha_{k+1}\beta_{k+1}=\alpha_k\beta_k=\frac{\alpha_k^2(1+\sigma^2)}{\gamma_k}\le \frac{1}{(1+\sigma^2)L}$, we have
			\begin{equation}\label{SGDdecay1}
				\mathbb{E}\left[\mathcal{E}(z_{k+1}^+;\gamma_{k+1})-\mathcal{E}(z_{k+1};\gamma_{k+1})\right] \leq - \mathbb{E}\left[ \frac{\alpha_k^2(1+\sigma^2)}{2\gamma_k}\| \nabla f (x_{k+1}) \|^2 \right].
			\end{equation}
			Thus the shift from $x_{k+1}$ to $x_{k+1}^+$ gives the required descent.
			
			Using the definition of Bregman divergence and the schedule for $\gamma_k$, we split the second term in~\eqref{equ147} as
			\begin{equation}\label{equ145}
				\begin{aligned}
					&\mathcal{E}(z_{k+1}; \gamma_{k+1}) - \mathcal{E}(z_{k}^+; \gamma_{k}) \\
					&= \mathcal{E}(z_{k+1}; \gamma_{k}) - \mathcal{E}(z_{k}^+; \gamma_{k}) + \frac{\gamma_{k+1} - \gamma_{k}}{2} \| v_{k+1} - x^{\star} \|^2 \\
					&\le \langle \nabla \mathcal{E}(z_{k+1}; \gamma_{k}), z_{k+1} - z_k^+ \rangle - D_{\mathcal{E}} (z_k^+, z_{k+1}; \gamma_{k})+ \frac{\alpha_{k}(\mu-\gamma_{k+1})}{2} \| v_{k+1} - x^{\star} \|^2.
				\end{aligned}
			\end{equation}
			Expanding the first term and using~\eqref{SHANG} gives
			\begin{equation}\label{equ146}
				\begin{aligned}
					&	\langle \nabla \mathcal{E}(z_{k+1}; \gamma_{k}), z_{k+1} - z_k^+ \rangle\\
					={}			& -\alpha_k \langle \nabla f(x_{k+1}) - \nabla f(x^{\star}), x_{k+1} - x^{\star} \rangle + \alpha_k \langle g(x_{k+1}), v_k - v_{k+1} \rangle \\
					& - \frac{\alpha_k\mu}{2} \| v_{k+1} - x^{\star} \|^2 - \frac{\alpha_k\mu}{2} \| x_{k+1} - v_{k+1} \|^2 + \frac{\alpha_k\mu}{2} \| x_{k+1} - x^{\star} \|^2 \\
					& + \alpha_k \langle \nabla f(x_{k+1}) - g (x_{k+1}), v_k - x^{\star}  \rangle .
				\end{aligned}
			\end{equation}
			After conditional expectation $\mathbb E[\cdot \mid \mathcal F_{k+1}]$, the last term vanishes. The positive term $\frac{\alpha_k\mu}{2} \| x_{k+1} - x^{\star} \|^2$ is controlled by the strong convexity part of
			$$
			-\alpha_k \langle \nabla f(x_{k+1}) - \nabla f(x^{\star}), x_{k+1} - x^{\star} \rangle
			= -\alpha_k \left(D_f(x_{k+1}, x^{\star}) + D_f(x^{\star}, x_{k+1})\right).
			$$
			The key stochastic term is the cross term involving $g(x_{k+1})$ and $v_{k+1}$. By Cauchy--Schwarz and Young's inequalities,
			\begin{equation}
				\begin{aligned}
					\alpha_k \langle g(x_{k+1}), v_k - v_{k+1} \rangle
					&\le \frac{\alpha_k^2}{2\gamma_{k}} \| g(x_{k+1})\|^2 + \frac{\gamma_{k}}{2} \| v_k - v_{k+1} \|^2 .
				\end{aligned}
			\end{equation}
			The term $\frac{\gamma_k}{2}\|v_k-v_{k+1}\|^2$ is canceled by $-D_{\mathcal{E}}(z_k^+,z_{k+1};\gamma_k)$. The MNS bound in Lemma \ref{lemma1} gives
			$$
			\frac{\alpha_k^2}{2\gamma_k} \mathbb{E}\left[\|g(x_{k+1})\|^2 \mid \mathcal{F}_{k+1} \right] \leq \frac{\alpha_k^2(1+\sigma^2)}{2\gamma_k}\|\nabla f(x_{k+1})\|^2,
			$$
			which is canceled by~\eqref{SGDdecay1}. The remaining negative terms give
			$$
			-\alpha_k\mathcal{E}(z_{k+1};\gamma_{k+1})\leq -\alpha_k\mathcal{E}(z_{k+1}^+;\gamma_{k+1}).
			$$
			Rearranging yields
			$$
			\mathbb{E}\left[\mathcal{E}(z_{k+1}^+;\gamma_{k+1})\right]\leq \frac{1}{1+\alpha_k}\mathbb{E}\left[ \mathcal{E}(z_k^+;\gamma_k) \right] \le \prod_{i=0}^{k} \frac{1}{1+\alpha_i} \mathbb{E}[\mathcal{E}(z_0^+;\gamma_0)].
			$$
		\end{proof}
		
		When $\sigma=0$, SHANG reduces to the deterministic HNAG method analyzed in~\cite{ChenLuo2021}. As $\sigma$ grows, convergence slows but acceleration is preserved through the noise-dependent stepsize. Theorem~\ref{theorem1} implies convergence in expectation of $f(x_k^+)$ to $f(x^\star)$. It also implies almost-sure convergence by Markov's inequality and the Borel--Cantelli lemma; see Section~\ref{corollaryanalysis}.
		\begin{corollary}\label{corollary1}
			In the setting of Theorem~\ref{theorem1}, $f(x_k^+) \overset{a.s.}{\rightarrow} f(x^{\star})$.
		\end{corollary}
		
		\subsection{SHANG++}\label{sec:shang++}
		In SHANG, the $x$- and $v$-updates use the same stepsize $\alpha_k$, which couples the coefficients in the two equations. This shared scaling is natural from the underlying dynamics, but more flexibility is useful under MNS, where the effective constants are rescaled.
		
		We therefore use asymmetric stepsize scalings: a reduced effective stepsize $\tilde{\alpha}_k$ in the $x$-update and $\alpha_k$ in the $v$-update. The time-scaling sequence $\{\gamma_k\}$ will be specified later. This gives SHANG++, a faster and more noise-robust variant of SHANG:
		\begin{subequations}\label{SHANG++}
			\begin{align}
				\frac{x_{k+1}-x_k}{\tilde{\alpha}_k}
				&= v_k - x_{k+1}  - \beta_k g(x_k) \label{SHANG++:a}\\
				\frac{v_{k+1}-v_k}{\alpha_k}
				&= \frac{\mu}{\gamma_k}(x_{k+1}-v_{k+1}) - \frac{1}{\gamma_k}g(x_{k+1}) \label{SHANG++:b}
			\end{align}
		\end{subequations}
		initialized with $x_0$ and $v_0$, where $\tilde{\alpha}_k := \frac{\alpha_k}{1+m\alpha_k}$ and $m\ge 0$. Equivalently, \eqref{SHANG++:a} can be written as
		\begin{equation}\label{equ23}
			\frac{x_{k+1}-x_k}{\alpha_k}
			= v_k - x_{k+1} - m(x_{k+1}-x_k) - \beta_k g(x_k).
		\end{equation}
		Thus SHANG++ adds the correction term $-m(x_{k+1}-x_k)$ to the $x$-update of SHANG~\eqref{SHANG:a}. The parameter $m$ controls the correction strength, and $m=0$ recovers SHANG. The notation ``++'' indicates that this extra degree of freedom gives stronger convergence guarantees and improved robustness in the reported experiments.
		
		\subsubsection{SHANG++ for Strongly Convex Minimization}
		Assume $f\in\mathcal{S}_{\mu,L}$ with $0<\mu<L<\infty$. We consider SHANG++ with constant parameters $\gamma_k\equiv \mu$ and $\alpha_k\equiv \alpha$, and choose $m=1$, so that $\tilde{\alpha}=\frac{\alpha}{1+\alpha}$. Define
		\begin{equation}\label{auxi2}
			x_{k}^+ := x_{k} - \tilde{\alpha} \beta g(x_{k}).
		\end{equation}
		Then SHANG++~\eqref{SHANG++} can be written as
		\begin{subequations}\label{equ26}
			\begin{align}
				\frac{x_{k+1} - x_k^+}{\tilde{\alpha}} &= v_k - x_{k+1} \label{equ26:a}\\
				\frac{v_{k+1} - v_k}{\alpha} & = x_{k+1} - v_{k+1} - \tfrac{1}{\mu} g(x_{k+1}) \label{equ26:b} \\
				x_{k+1}^+ & = x_{k+1} - \tilde{\alpha} \beta g(x_{k+1})  \label{equ26:c}
			\end{align}
		\end{subequations}
		initialized with $x_0$ and $v_0$.
		
		\begin{theorem}\label{theorem2}
			Let $f\in\mathcal{S}_{\mu,L}$. Starting from $x_0=v_0$, generate $(x_k,v_k)$ by~\eqref{equ26}. Let the stochastic gradients be conditionally unbiased and satisfy the MNS condition~\eqref{eq:mns}. Define $x_k^+$ by~\eqref{auxi2}. Choose $\alpha=\frac{\tilde{\alpha}}{1-\tilde{\alpha}}$ and $\beta=\frac{\tilde{\alpha}}{\mu/(1+\sigma^2)}$ with $0<\tilde{\alpha}\le \frac{1}{1+\sigma^2}\sqrt{\frac{\mu}{L}}$. Then for all $k\ge 0$,
			$$
			\mathbb{E}\left[f(x_{k+1}^+) - f(x^{\star}) + \frac{\mu}{2} \| v_{k+1} - x^{\star} \|^2 \right]
			\le (1 - \tilde\alpha)^{k+1}\mathbb{E}[  \mathcal{E}(z_0^+; \mu)].
			$$
		\end{theorem}
		
		Theorem~\ref{theorem2} shows that $\mathbb{E}[f(x_{k+1}^+) - f(x^{\star})]$ contracts linearly at rate $\mathcal{O}\left((1 - \frac{1}{1+\sigma^{2}}\sqrt{\mu/L})^{k+1}\right)$, which is sharper than $\mathcal{O}\left((1 + \frac{1}{1+\sigma^{2}}\sqrt{\mu/L})^{-(k+1)}\right)$ from Theorem~\ref{theorem1}.
		
		In the deterministic case $\sigma=0$, the choice $m=\beta\mu$ recovers the HNAG++ method in~\cite{ChenXu2025}. Since $\beta\mu=\tilde{\alpha}\le \sqrt{\mu/L}$ in that case, the choice $m=1$ used in Theorem~\ref{theorem2} allows a stronger correction when $\mu/L$ is small. This extra freedom is useful in both tuning and analysis.
		
		We give a proof outline and defer the full proof to Section~\ref{shang++scanalysis}. 
		\begin{proof}[Proof outline]
			Introduce $\mathcal{E}(z_{k+1};\mu)$ as an intermediate quantity. Let $\mathcal F_{k+1}$ be the sigma-algebra generated by the information available before sampling $g(x_{k+1})$. Then $x_k^+$, $x_{k+1}$, and $v_k$ are $\mathcal F_{k+1}$-measurable, while $v_{k+1}$ is determined after $g(x_{k+1})$ is sampled. 
			
			By Lemma~\ref{lemma2}, applied conditionally to
			$$
			x_{k+1}^+=x_{k+1}-\tilde{\alpha}\beta g(x_{k+1}),
			$$
			we have
			\begin{equation}\label{eq:SDGdecay}
				\mathbb{E}\left[f(x_{k+1}^+)-f(x_{k+1})\mid \mathcal F_{k+1}\right]
				\le
				-\frac{\tilde{\alpha}\beta}{2}\|\nabla f(x_{k+1})\|^2
				=
				-\frac{\tilde{\alpha}^2(1+\sigma^2)}{2\mu}\|\nabla f(x_{k+1})\|^2 .
			\end{equation}
			Here we used $\beta=\tilde{\alpha}(1+\sigma^2)/\mu$ and $\tilde{\alpha}\beta$ satisfies the requirement in Lemma~\ref{lemma2}. 
			Expanding $\mathcal{E}(z_{k+1};\mu)-\mathcal{E}(z_k^+;\mu)$ and using~\eqref{equ26}, we obtain
			\begin{equation}\label{equ102}
				\begin{aligned}
					&\mathcal{E}(z_{k+1};\mu)-\mathcal{E}(z_k^+;\mu) \\
					={}& -\alpha \langle \nabla f(x_{k+1}) - \nabla f(x^{\star}), x_{k+1} - x^{\star} \rangle  -D_f(x_k^+,x_{k+1})-\frac{\mu}{2}\|v_k-v_{k+1}\|^2 \\
					&- \frac{\alpha\mu}{2} \| v_{k+1} - x^{\star} \|^2 - \frac{\alpha\mu}{2} \| v_{k+1} - x_{k+1} \|^2+ \frac{\alpha\mu}{2} \| x_{k+1} - x^{\star} \|^2 \\
					&+ \tilde{\alpha} \langle g(x_{k+1}), v_k - v_{k+1} \rangle + \alpha\tilde{\alpha} \langle g(x_{k+1}), x_{k+1} - v_{k+1} \rangle\\
					& +\tilde{\alpha}\langle \nabla f(x_{k+1})-g(x_{k+1}),v_k-x_{k+1}\rangle+\alpha\langle \nabla f(x_{k+1})-g(x_{k+1}),x_{k+1}-x^\star\rangle .
				\end{aligned}
			\end{equation}
			Compared with SHANG, the SHANG++ expansion has one extra stochastic cross term caused by the asymmetric stepsizes, and this changes the cancellation from a single auxiliary descent to a $(1+\alpha)$-weighted auxiliary descent.
			
			Taking conditional expectation with respect to $\mathcal F_{k+1}$, the last two terms vanish by conditional unbiasedness. Strong convexity gives
			$$
			-\alpha \langle \nabla f(x_{k+1})-\nabla f(x^\star),x_{k+1}-x^\star\rangle
			\le
			-\alpha\bigl(f(x_{k+1})-f(x^\star)\bigr)
			-\frac{\alpha\mu}{2}\|x_{k+1}-x^\star\|^2.
			$$
			Young's inequality gives
			$$
			\tilde{\alpha}\langle g(x_{k+1}),v_k-v_{k+1}\rangle
			\le
			\frac{\tilde{\alpha}^2}{2\mu}\|g(x_{k+1})\|^2
			+
			\frac{\mu}{2}\|v_k-v_{k+1}\|^2,
			$$
			and
			$$
			\alpha\tilde{\alpha}\langle g(x_{k+1}),x_{k+1}-v_{k+1}\rangle
			\le
			\frac{\alpha\tilde{\alpha}^2}{2\mu}\|g(x_{k+1})\|^2
			+
			\frac{\alpha\mu}{2}\|x_{k+1}-v_{k+1}\|^2.
			$$
			Using the conditional MNS bound, we obtain
			\begin{equation}\label{eq:cond-positive-term}
				\begin{aligned}
					&\mathbb{E}\left[\mathcal{E}(z_{k+1};\mu)-\mathcal{E}(z_k^+;\mu)\mid \mathcal F_{k+1}\right] \\
					&\le
					-\alpha\left(f(x_{k+1})-f(x^\star)+\frac{\mu}{2}\|v_{k+1}-x^\star\|^2\right)
					-D_f(x_k^+,x_{k+1}) \\
					&\quad
					+\frac{(1+\alpha)\tilde{\alpha}^2(1+\sigma^2)}{2\mu}
					\|\nabla f(x_{k+1})\|^2 .
				\end{aligned}
			\end{equation}
			The positive gradient-norm term in~\eqref{eq:cond-positive-term} is canceled by the descent in~\eqref{eq:SDGdecay} with factor $1+\alpha$. Indeed,
			$$
			(1+\alpha)\mathbb{E}\left[f(x_{k+1}^+)-f(x_{k+1})\mid\mathcal F_{k+1}\right]
			\le
			-\frac{(1+\alpha)\tilde{\alpha}^2(1+\sigma^2)}{2\mu}
			\|\nabla f(x_{k+1})\|^2.
			$$
			Combining this estimate with~\eqref{eq:cond-positive-term} gives
			$$
			\mathbb{E}\left[\mathcal{E}(z_{k+1}^+;\mu)-\mathcal{E}(z_k^+;\mu)\mid\mathcal F_{k+1}\right]
			\le
			-\alpha\mathbb{E}\left[\mathcal{E}(z_{k+1}^+;\mu)\mid\mathcal F_{k+1}\right].
			$$
			Hence
			$$
			\mathbb{E}\left[\mathcal{E}(z_{k+1}^+;\mu)\mid\mathcal F_{k+1}\right]
			\le
			\frac{1}{1+\alpha}\mathcal{E}(z_k^+;\mu)
			=
			(1-\tilde{\alpha})\mathcal{E}(z_k^+;\mu),
			$$
			where the last equality follows from $\alpha=\tilde{\alpha}/(1-\tilde{\alpha})$.
		\end{proof}
		
		While Gupta et al.~\cite{GuptaSiegelWojtowytsch2024} view noise as inflating smoothness to $(1+\sigma^2)L$, our analysis shows that it changes both smoothness and curvature, with $L_{\sigma} = (1+\sigma^2)L$ and $\mu_{\sigma} = \mu/(1+\sigma^2)$. Thus the admissible stepsize scaling differs between SHANG and SHANG++:
		$$
		\text{(SHANG)} \quad 0 < \alpha \le \sqrt{\frac{\mu_{\sigma}}{L_{\sigma}}},
		\qquad
		\text{(SHANG++)} \quad 0 < \alpha \le \frac{1}{1-\tilde{\alpha}}\sqrt{\frac{\mu_{\sigma}}{L_{\sigma}}}.
		$$
		The noise-damping term in SHANG++ reduces the effective Lipschitz constant from $L_{\sigma}$ to $(1-\tilde{\alpha})L_{\sigma}$ and increases the effective strongly convex constant from $\mu_{\sigma}$ to $\mu_{\sigma}/(1-\tilde{\alpha})$. This helps explain the improved stability and convergence guarantees of SHANG++.
		
		\subsubsection{SHANG++ Method for Convex Minimization}
		Assume $\mu = 0$ and $m \ge 0$. To streamline the analysis, introduce an effective time-scaling sequence $\{\tilde{\gamma}_k\}$ satisfying $\frac{\gamma_k}{\tilde{\gamma}_k} = \frac{\alpha_k}{\tilde{\alpha}_k} = 1+m \alpha_k$. We choose
		$$
		\alpha_k = \frac{2}{k+1}, \quad \gamma_k = \alpha_k \tilde{\alpha}_k (1+\sigma^2)^2L, \quad \beta_k = \frac{\alpha_k(1+\sigma^2)}{\gamma_k}.
		$$
		Equivalently,
		$$
		\tilde{\alpha}_k = \frac{2}{k+1+2m}, \quad \tilde{\gamma}_k = \tilde{\alpha}_k^2(1+\sigma^2)^2L, \quad \beta_k = \frac{\tilde{\alpha}_k(1+\sigma^2)}{\tilde{\gamma}_k}.
		$$
		With these schedules, SHANG++~\eqref{SHANG++} can be written as
		\begin{subequations}\label{equ120}
			\begin{align}
				\frac{x_{k+1} - x_k}{\tilde{\alpha}_k} &= v_k - x_{k+1} - \beta_k g(x_k) \label{equ120:a}\\
				\frac{v_{k+1} - v_k}{\tilde{\alpha}_k} &= -\frac{1}{\tilde{\gamma}_{k}} g(x_{k+1}) \label{equ120:b}\\
				\frac{\tilde{\gamma}_{k+1} - \tilde{\gamma}_k}{\tilde{\alpha}_k} &\le - \tilde{\gamma}_{k+1} \label{equ120:c}
			\end{align}
		\end{subequations}
		initialized with $x_0$, $v_0$, and $\tilde{\gamma}_0$.
		
		\begin{theorem}\label{theorem3}
			Let $f\in\mathcal{S}_{0,L}$. Starting from $x_0=v_0$, generate $(x_k,v_k)$ by~\eqref{SHANG++}. Let the stochastic gradients be conditionally unbiased and satisfy the MNS condition~\eqref{eq:mns}. Define $x_k^+ = x_k - \tilde{\alpha}_k\beta_k g(x_k)$. Assume $m \ge 0$, $\alpha_k = \frac{2}{k+1}$, $\gamma_k/(1+\sigma^2)  = \alpha_k \tilde{\alpha}_k L_{\sigma}$, and $\beta_k = \frac{\alpha_k}{\gamma_k/(1+\sigma^2)}$. Then for all $k\ge 0$,
			$$
			\mathbb{E}\left[f(x_{k+1}^+) - f(x^{\star}) + \frac{\tilde{\gamma}_{k+1}}{2} \| v_{k+1} - x^{\star} \|^2 \right] \le \frac{(1+2m)(2+2m)}{(k+2+2m)(k+3+2m)}\mathcal{E}_0,
			$$
			where $\mathcal{E}_0 = \mathbb{E}[ \mathcal{E}(z_0^+; \tilde{\gamma}_0)]$.
		\end{theorem}
		
		The system~\eqref{equ120} is the SHANG recursion with $(\alpha_k,\gamma_k)$ replaced by $(\tilde{\alpha}_k,\tilde{\gamma}_k)$. Applying Theorem~\ref{theorem1} gives
		$$
		\mathbb{E}\left[\mathcal{E}(z_{k+1}^+;\tilde{\gamma}_{k+1})\right]
		\le
		\prod_{i=0}^{k}\frac{1}{1+\tilde{\alpha}_i}\mathcal{E}_0.
		$$
		Since $\tilde{\alpha}_i=\frac{2}{i+1+2m}$,
		$$
		\prod_{i=0}^{k}\frac{1}{1+\tilde{\alpha}_i}
		=
		\prod_{i=0}^{k}\frac{i+1+2m}{i+3+2m}
		=
		\frac{(1+2m)(2+2m)}{(k+2+2m)(k+3+2m)}.
		$$
		This proves Theorem~\ref{theorem3}.
		
		\begin{corollary}\label{corollary3}
			Under the setting of Theorem~\ref{theorem2} or~\ref{theorem3}, $f(x_k^+) \overset{a.s.}{\rightarrow} f(x^{\star})$.
		\end{corollary}
		
		We next compare the stepsize-dependent scaling induced by the two methods in the convex setting:
		$$
		\text{(SHANG)}\quad \frac{\gamma_k}{1+\sigma^2}=\alpha_k^2L_\sigma,
		\qquad
		\text{(SHANG++)}\quad \frac{\gamma_k}{1+\sigma^2}=\alpha_k^2\cdot \frac{L_\sigma}{1+m\alpha_k}.
		$$
		Thus SHANG++ replaces $L_\sigma$ by $L_\sigma/(1+m\alpha_k)$ in the scaling, which partly offsets the $\sigma^2$-dependent amplification in the admissible stepsize. The effect of $m$ is not monotone: larger $m$ relaxes the noise-dependent stepsize restriction but may weaken the convex rate. This gives a trade-off between speed and robustness. In our experiments, $m=1$ and $m=1.5$ work well.
		
		\paragraph{Batching}
		Gradient noise can be reduced by increasing the mini-batch size $M$ in~\eqref{equsg}. If $\sigma_1^2$ is the MNS constant for $M=1$, then $\sigma_M^2 \le \sigma_1^2/M$. Another option is to average $K$ independent gradient estimators,
		$$
		g^K = \frac{1}{K} \sum_{i=1}^K g_i,
		$$
		which gives an effective MNS constant $\sigma^2/K$. Both approaches reduce noise at the cost of higher computation.
		
		\subsection{SNAG as a Discretization of the HNAG Flow}\label{sectionSNAG}
		Under MNS, one recent first-order stochastic method designed to overcome the divergence of NAG and accelerate SGD is the Stochastic Nesterov Accelerated Gradient (SNAG) method~\cite{Nesterov2013,HermantEtAl2025}. Its iteration is
		\begin{equation}\label{SNAG}
			\begin{aligned}
				x_{k+1}  &= \hat{\alpha}_{k+1} x_k + (1-\hat{\alpha}_{k+1})v_{k+1} - \hat{\alpha}_{k+1} s\, g(x_k), \\
				v_{k+1} & = \hat{\beta} v_k + (1-\hat{\beta}) x_k - \eta_k g(x_k),
			\end{aligned}
		\end{equation}
		where $g(x_k)$ is a stochastic gradient estimator, and $\hat{\alpha}_{k+1}$, $s$, $\hat{\beta}$, and $\eta_k$ are parameters.
		
		Set $\hat{\alpha}_{k+1} = \frac{1}{1+\alpha_{k+1}}$, $s = \alpha_{k+1}\beta_{k+1}$, $\hat{\beta} = \frac{1}{1+\frac{\alpha_{k+1}\mu}{\gamma_{k+1}}}$, and $\eta_k = \frac{1}{1+\frac{\alpha_{k+1}\mu}{\gamma_{k+1}}}\frac{\alpha_{k+1}}{\gamma_{k+1}}$. Then the SNAG scheme~\eqref{SNAG} is equivalent to
		\begin{equation}\label{newSNAG}
			\begin{aligned}
				\frac{x_{k+1} - x_k}{\alpha_{k+1}} & = v_{k+1} - x_{k+1} - \beta_{k+1} g(x_k), \\
				\frac{v_{k+1} - v_k}{\alpha_{k+1}} & = \frac{\mu}{\gamma_{k+1}}(x_k - v_{k+1}) - \frac{1}{\gamma_{k+1}} g(x_k), \\
				\frac{\gamma_{k+1} - \gamma_k}{\alpha_{k+1}} & \le \mu - \gamma_{k+1}.
			\end{aligned}
		\end{equation}
		Hence, SNAG can be interpreted as another discretization of the HNAG flow with more terms treated explicitly.
		
		A direct calculation shows that the parameter rules proposed for SNAG in~\cite{HermantEtAl2025} are compatible with those arising from HNAG discretizations. In our experiments, SNAG remains comparatively stable as the noise level grows, which is consistent with this HNAG-based interpretation.
		
		\section{Numerical Experiments}
		\label{test}
		The algorithm and convergence analysis in the previous section assume a convex objective. In modern deep learning, the training loss is typically nonconvex, so we adapt SHANG++ to a practical optimizer rather than a method with guaranteed global rates.
		
		\begin{algorithm}[H]
			\SetAlgoNoLine
			\SetNlSty{}{}{}
			\DontPrintSemicolon
			\caption{SHANG++ for Deep Learning}
			\label{algorithm1}
			\KwIn{Objective function $f$, initial point $x_0$, step size $\alpha$, time-scaling factor $\gamma$, noise damping $m$, iteration horizon $T$.}
			$k \gets 1$, $v_0 \gets x_0$, $x_1 \gets x_0$, $\tilde{\alpha} \gets \frac{\alpha}{1+m\alpha}$\;
			\While{$k \le T$}{
				Sample an unbiased stochastic gradient $g_k$ at $x_k$\;
				$v_k \gets v_{k-1} - \frac{\alpha}{\gamma} g_k$\;
				$x_{k+1} \gets \frac{1}{1 + \tilde{\alpha}}x_k + \frac{\tilde{\alpha}}{1 + \tilde{\alpha}}v_k - \frac{\tilde{\alpha}}{1 + \tilde{\alpha}} \frac{\alpha}{\gamma} g_k$\;
				$k \gets k + 1$\;
			} 
			\Return{$x_{T+1}$}
		\end{algorithm}
		
		For deep-learning experiments, we run SHANG++ with three explicit hyperparameters $(\alpha,\gamma,m)$ and set $\mu=0$ with the coupling $\beta=\alpha/\gamma$; see Algorithm~\ref{algorithm1}. The practical implementation uses an index-shifted form of~\eqref{SHANG++}. Instead of evaluating stochastic gradients at both $x_k$ and $x_{k+1}$, we first update $v_k$ using the stochastic gradient sampled at $x_k$, and then use the same gradient to compute $x_{k+1}$. Thus Algorithm~\ref{algorithm1} requires only one stochastic gradient evaluation per iteration. Apart from this gradient evaluation, the extra operations are vector additions and scalar-vector multiplications, hence $O(d)$ for $d$ parameters. No Hessian evaluation, Hessian-vector product, or linear solve is required. When $m=0$, Algorithm~\ref{algorithm1} reduces to SHANG, so the same per-iteration cost comparison applies to SHANG.
		
		Fixing $\beta=\alpha/\gamma$ reduces tuning to $(\alpha,\gamma,m)$ and works reliably across tasks. The convex theory suggests $\beta=(1+\sigma^2)\alpha/\gamma$, but the effective noise level $\sigma^2$ is difficult to estimate in practice and can vary over time. We therefore avoid explicit noise calibration.
		
		Throughout this section, NAG denotes the stochastic variant of Nesterov's accelerated gradient~\cite{Nesterov1983}, with $\nabla f(x)$ replaced by $g(x)$. SNAG refers to the method of~\cite{HermantEtAl2025}. SHB denotes the stochastic heavy-ball method, i.e., SGD with momentum.
		
		\begin{figure}[!htbp]
			\centering
			\begin{subfigure}{0.32\textwidth}
				\centering
				\includegraphics[width=\linewidth]{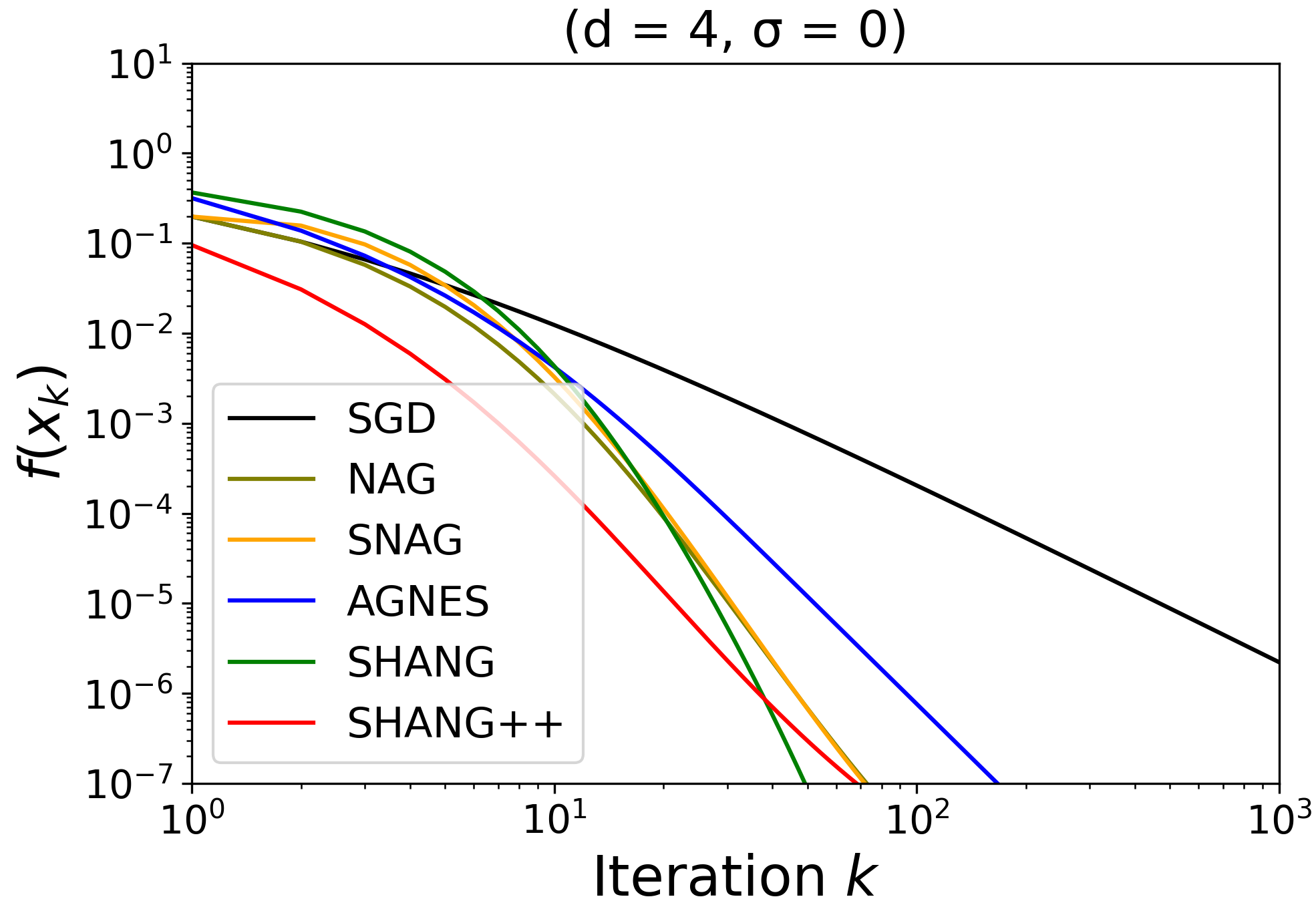}
			\end{subfigure}\hfill
			\begin{subfigure}{0.32\textwidth}
				\centering
				\includegraphics[width=\linewidth]{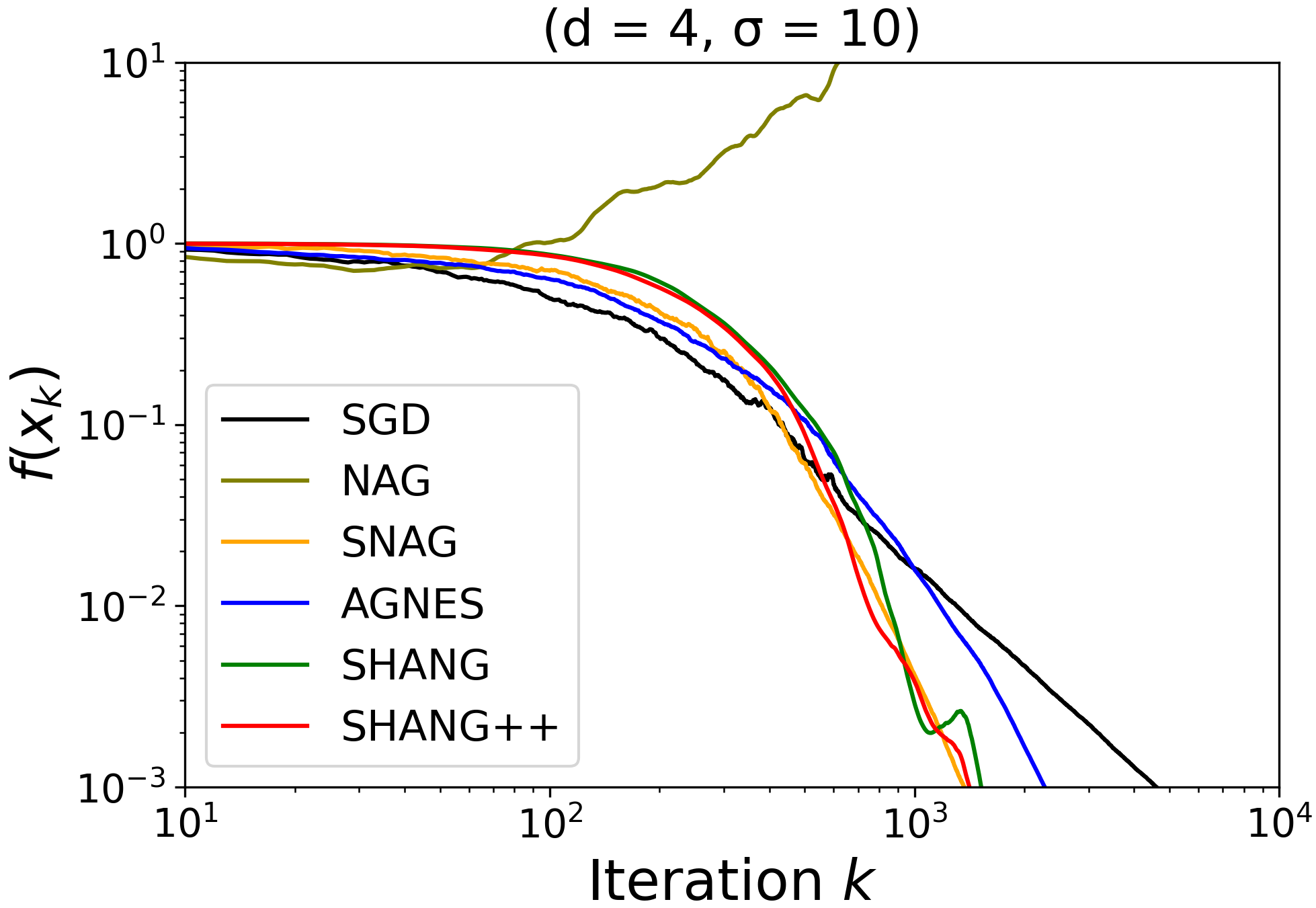}
			\end{subfigure}\hfill
			\begin{subfigure}{0.32\textwidth}
				\centering
				\includegraphics[width=\linewidth]{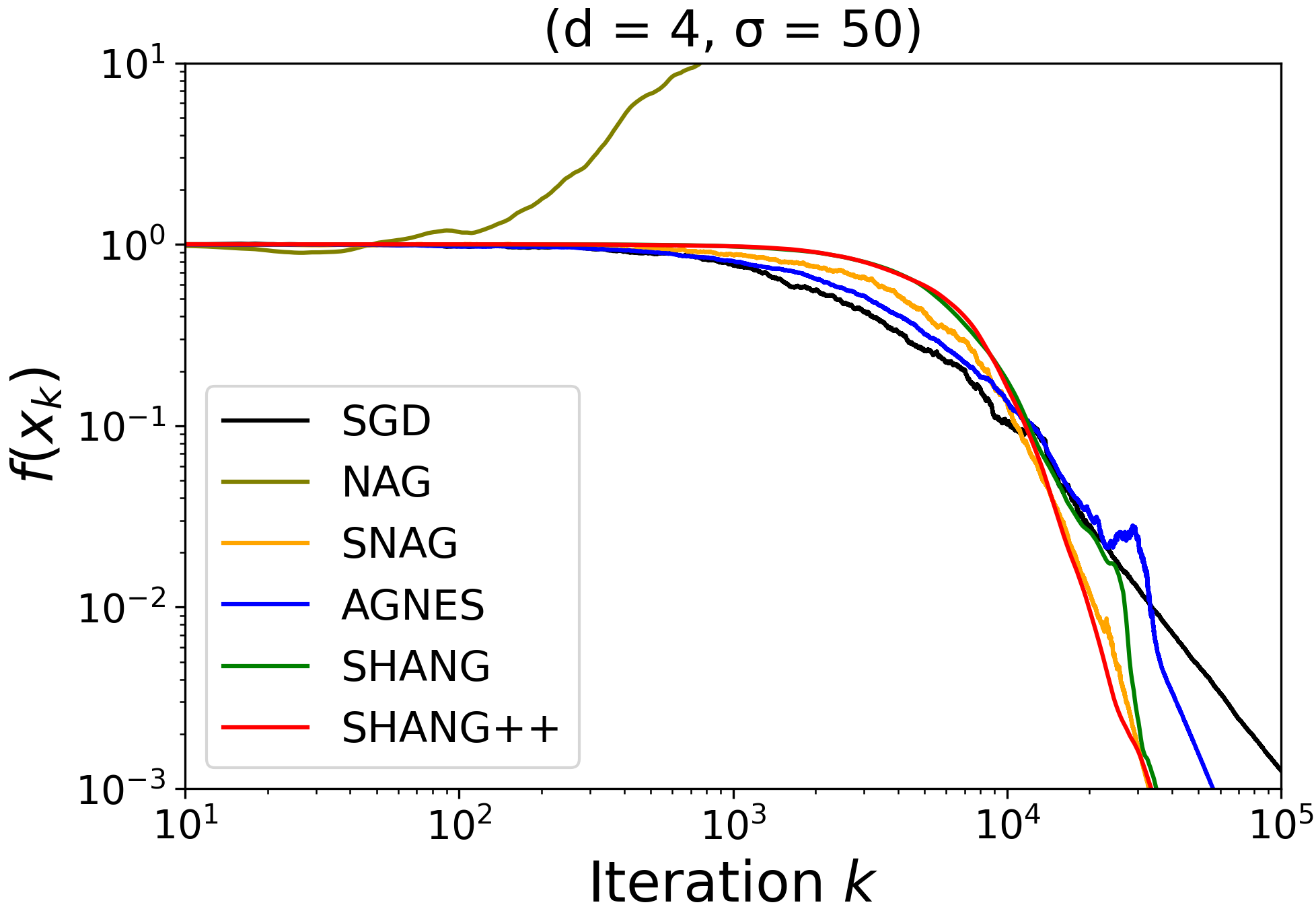}
			\end{subfigure}
			\vspace{0.5em}
			\begin{subfigure}{0.32\textwidth}
				\centering
				\includegraphics[width=\linewidth]{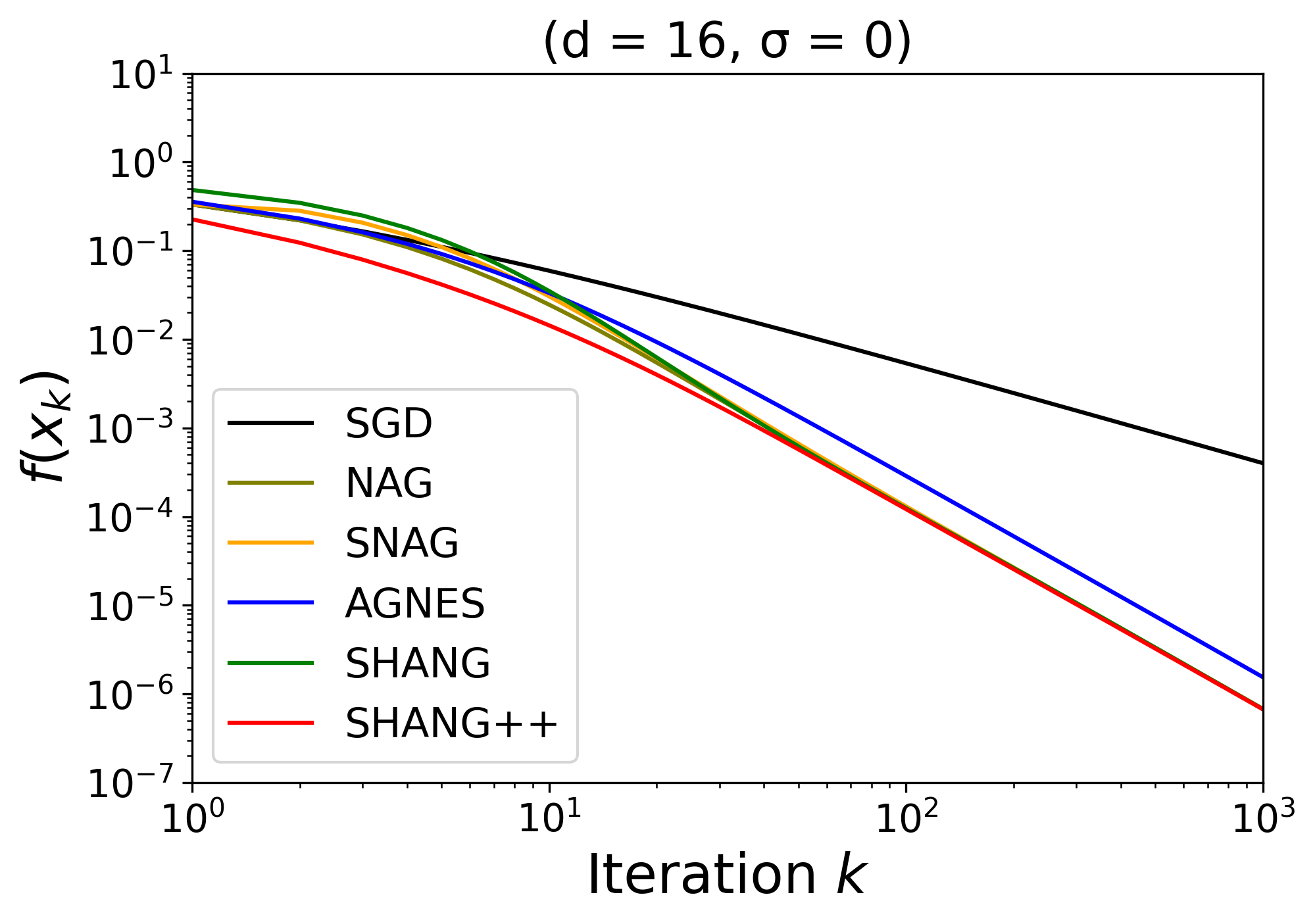}
			\end{subfigure}\hfill
			\begin{subfigure}{0.32\textwidth}
				\centering
				\includegraphics[width=\linewidth]{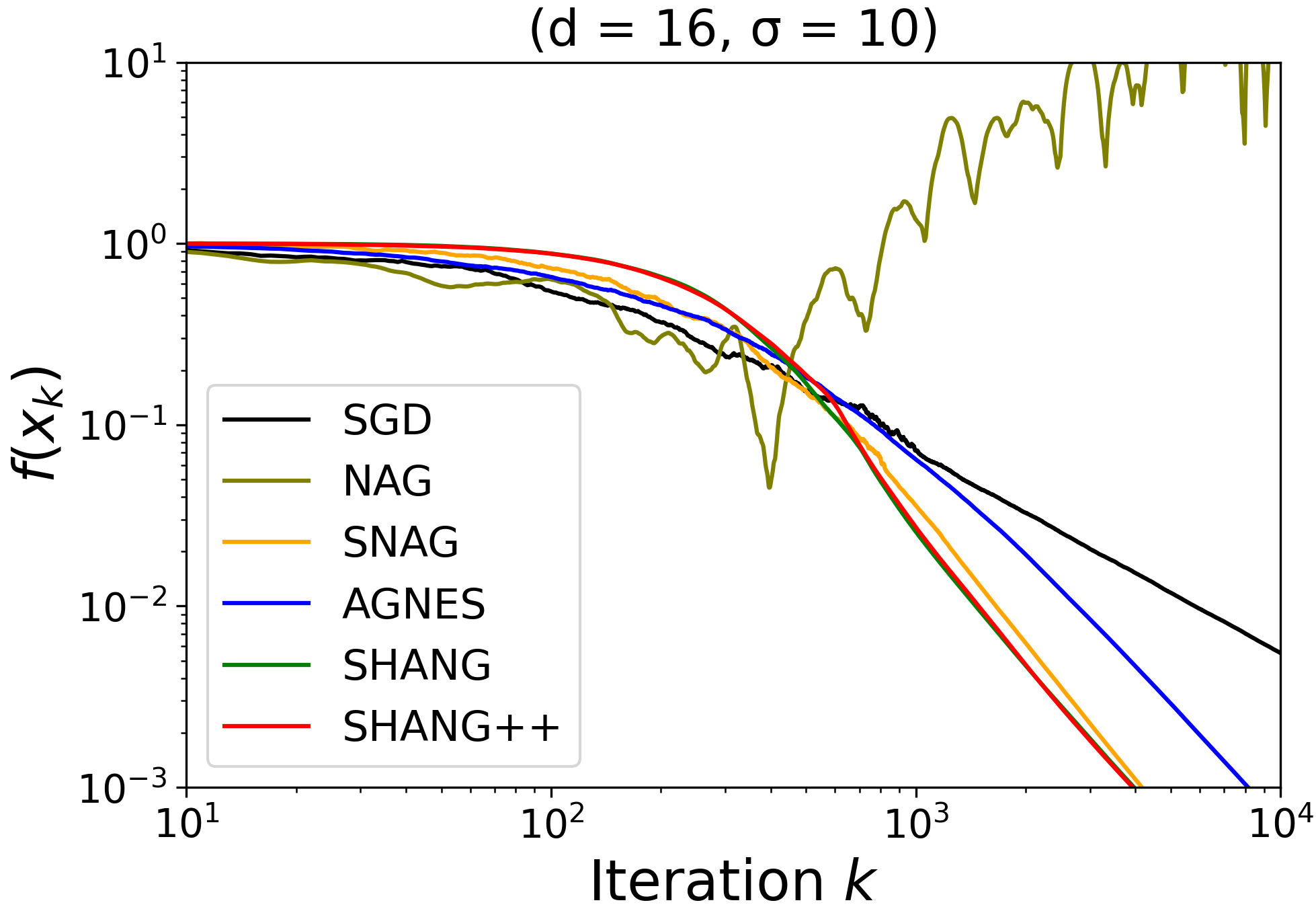}
			\end{subfigure}\hfill
			\begin{subfigure}{0.32\textwidth}
				\centering
				\includegraphics[width=\linewidth]{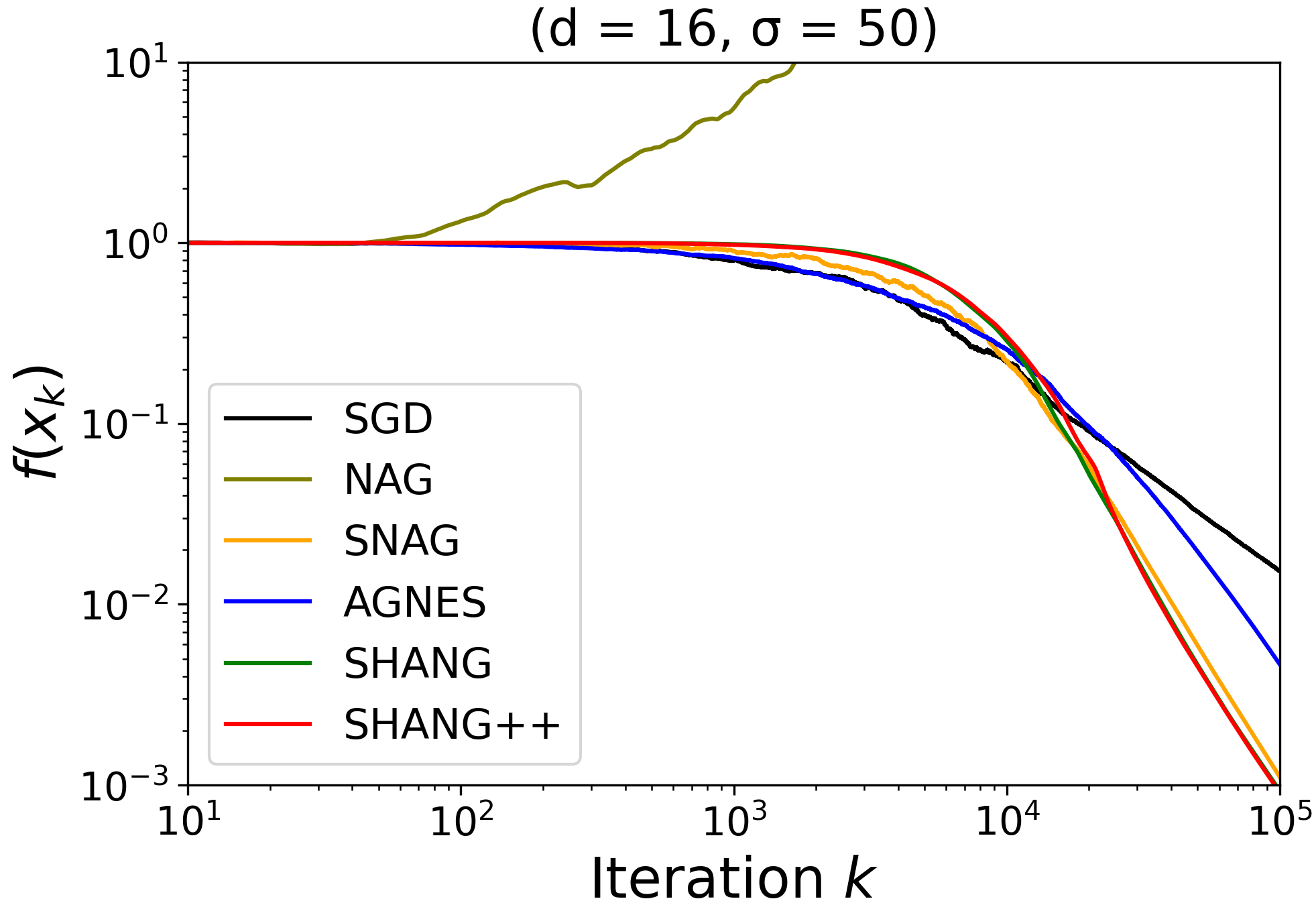}
			\end{subfigure}
			\caption{Performance of different algorithms under varying noise levels.}
			\label{fig:noise}
		\end{figure}
		
		\subsection{Convex optimization}\label{convextest}
		We first consider the family of objective functions from~\cite{GuptaSiegelWojtowytsch2024}:
		$$
		f_d : \mathbb{R} \to \mathbb{R}, \qquad 
		f_d(x) =
		\begin{cases} 
			|x|^d, & |x| < 1, \\[4pt]
			1+d(|x|-1), & \text{otherwise},
		\end{cases}
		$$
		for $d \ge 2$, with gradient estimators $g(x) = (1+ \sigma Z)\nabla f(x)$, where $Z \sim \mathcal{N}(0,1)$. The functions $f_d$ belong to $\mathcal{S}_{0,L}$ with $L = d(d-1)$.
		
		We compare SHANG and SHANG++ with SGD, NAG, AGNES~\cite{GuptaSiegelWojtowytsch2024}, and SNAG~\cite{HermantEtAl2025} under $\sigma \in \{0, 10, 50\}$ and $d \in \{4, 16\}$. The larger value $\sigma=50$ tests robustness in a high-noise regime. The parameters follow their optimal choices for the convex case. All simulations start from $x_0 =1$, and expectations are averaged over $200$ independent runs.
		
		In Figure~\ref{fig:noise}, both SHANG and SHANG++ remain stable as $\sigma$ increases, whereas NAG diverges under large noise. SHANG is very competitive, and SHANG++ is consistently slightly better than the other accelerated stochastic schemes. These results suggest that the proposed methods are robust to noisy gradients with modest tuning, while retaining accelerated-like behavior in the high-noise regime.
		
		\subsection{Classification Tasks on MNIST, CIFAR-10 and CIFAR-100}\label{classification}
		We benchmark three training tasks: LeNet-5 on MNIST~\cite{LeCunBottouBengioHaffner1998}, ResNet-34~\cite{HeZhangRenSun2016} on CIFAR-10~\cite{Krizhevsky2009}, and ResNet-50 on CIFAR-100 with standard data augmentation, including normalization, random crop, and random flip. The classification experiments test two aspects: performance across different tasks at fixed batch size, and the effect of batch size on stochastic-gradient noise. We first report MNIST, CIFAR-10, and CIFAR-100 results with batch size $50$, and then use CIFAR-10 for the batch-size ablation.
		
		\begin{figure}[!htbp]
			\centering
			\begin{subfigure}{0.45\textwidth}
				\centering
				\includegraphics[width=\linewidth]{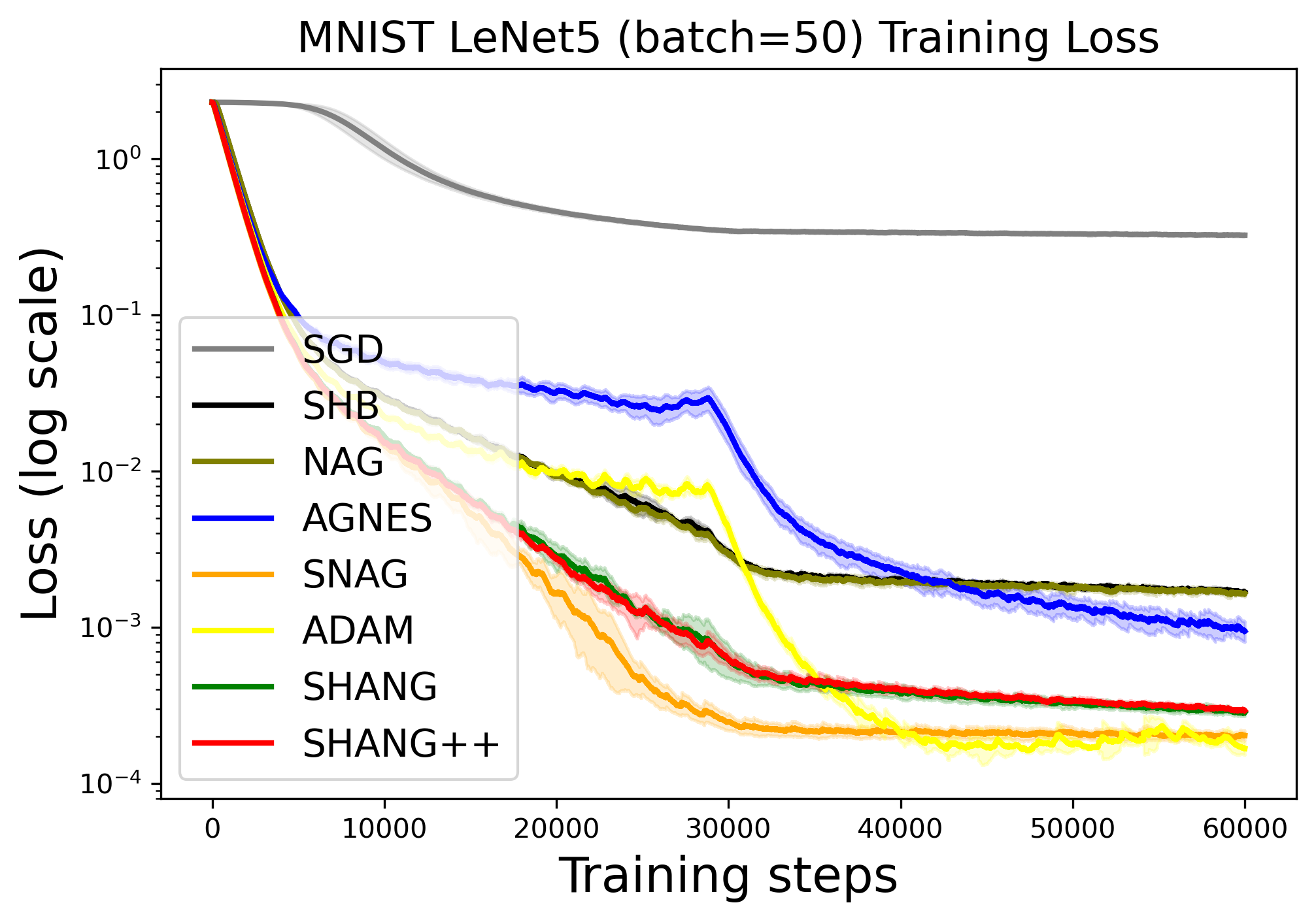}
			\end{subfigure}\hfill
			\begin{subfigure}{0.45\textwidth}
				\centering
				\includegraphics[width=\linewidth]{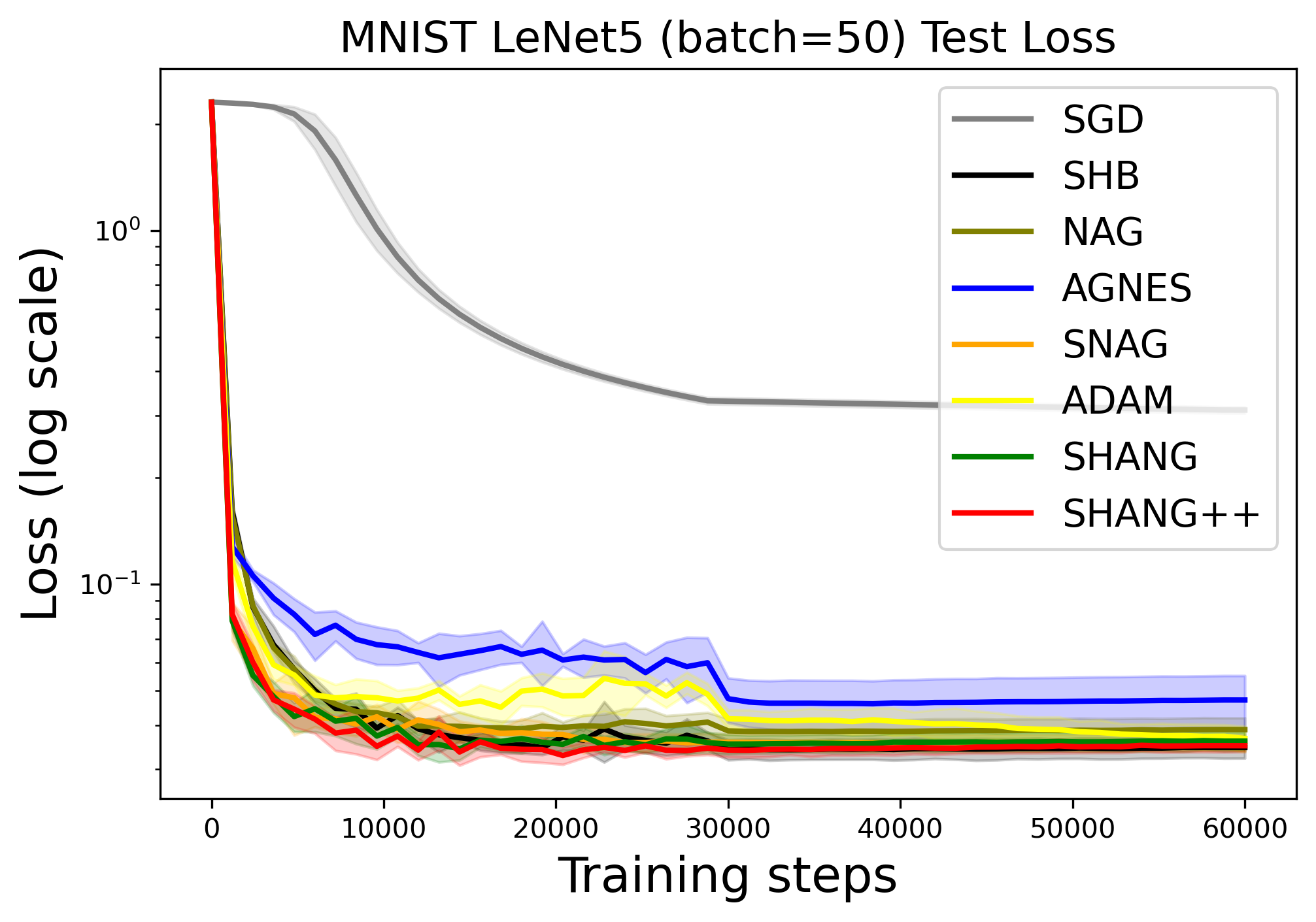}
			\end{subfigure}\hfill
			\caption{\small Training and test loss (log scale, running average with decay $0.99$) on MNIST with LeNet-5 (batch size $50$).}
			\label{MNIST-LeNet5}
		\end{figure}
		
		For each dataset, we run all algorithms for $50$ epochs and report averages over five trials. Following the AGNES experimental protocol in~\cite{GuptaSiegelWojtowytsch2024}, we use one mid-training schedule change at epoch $25$. After $25$ epochs, the learning rates for all baseline methods, except SHANG and SHANG++, are reduced by a factor of $0.1$; the AGNES correction step size is reduced similarly. For our methods, the time-scaling factor $\gamma$ is doubled after $25$ epochs. SHANG and SHANG++ do not use an explicit learning rate; their effective learning rate is controlled by $\gamma$, with effective learning rate factor $1/\gamma$; see Algorithm~\ref{algorithm1}. Thus increasing $\gamma$ after $25$ epochs gives a comparable mid-training reduction of the effective step size.
		
		\begin{figure}[!htbp]
			\centering
			\begin{subfigure}{0.45\textwidth}
				\centering
				\includegraphics[width=\linewidth]{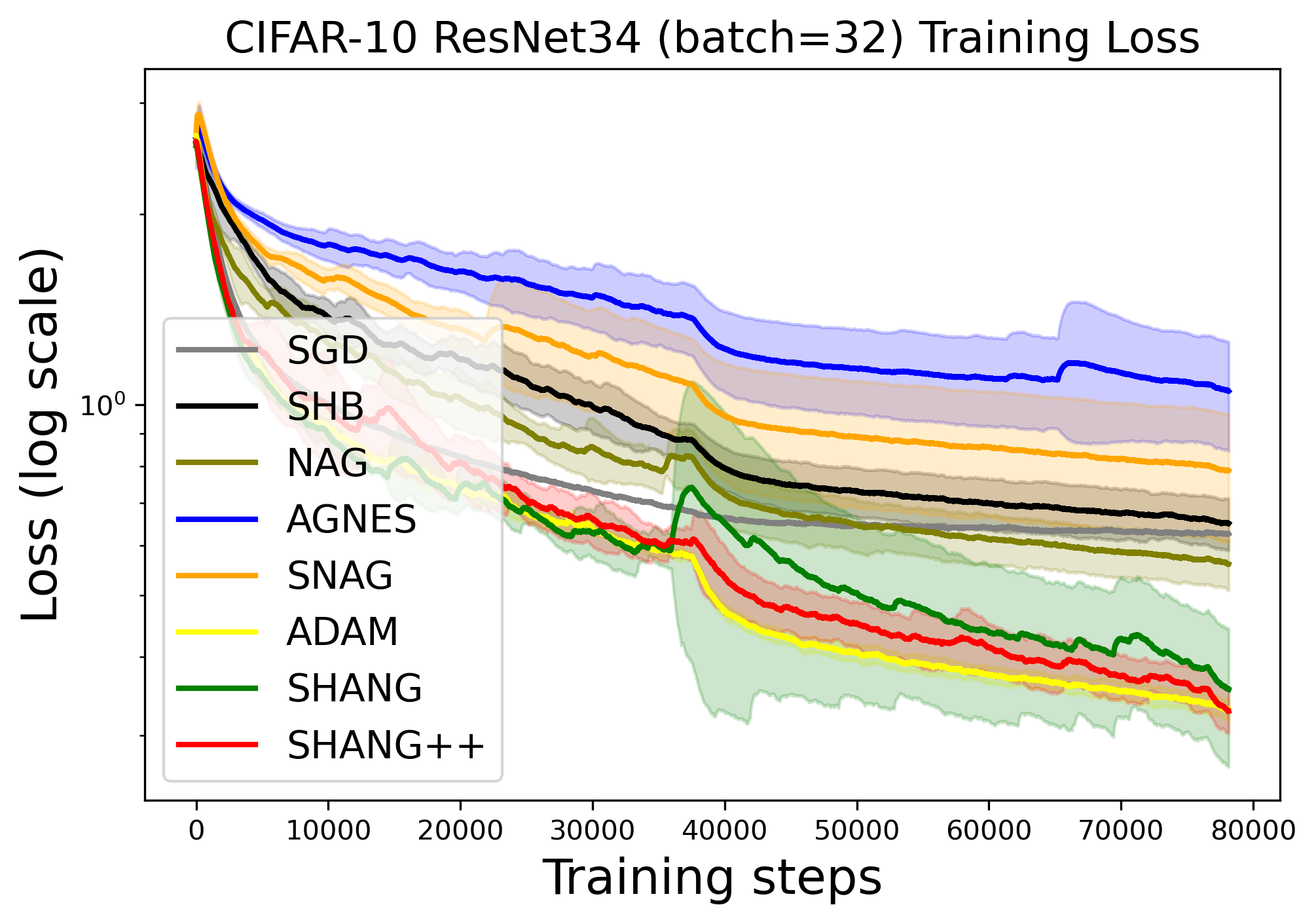}
			\end{subfigure}\hfill
			\begin{subfigure}{0.45\textwidth}
				\centering
				\includegraphics[width=\linewidth]{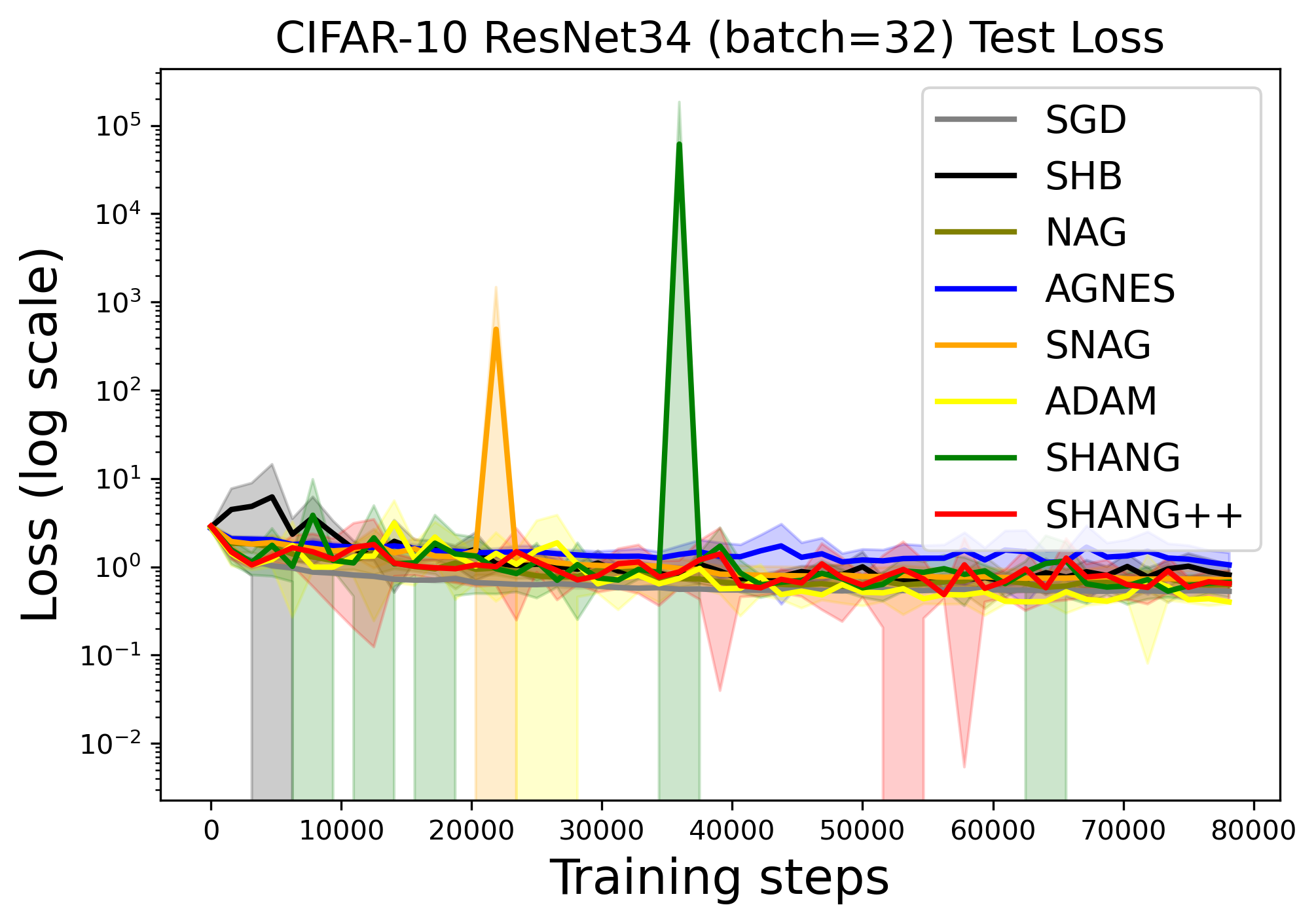}
			\end{subfigure}\hfill
			\begin{subfigure}{0.45\textwidth}
				\centering
				\includegraphics[width=\linewidth]{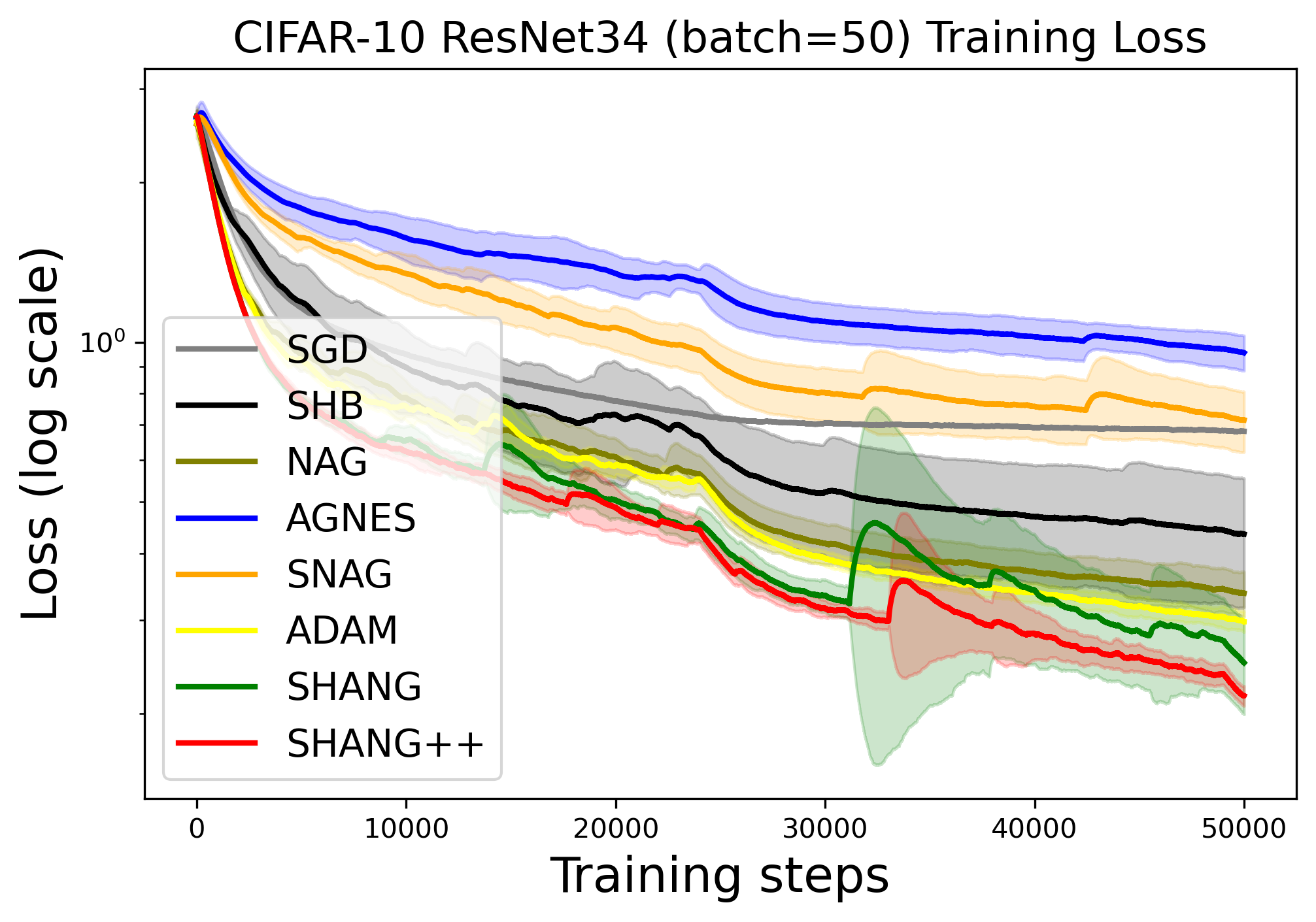}
			\end{subfigure}\hfill
			\begin{subfigure}{0.45\textwidth}
				\centering
				\includegraphics[width=\linewidth]{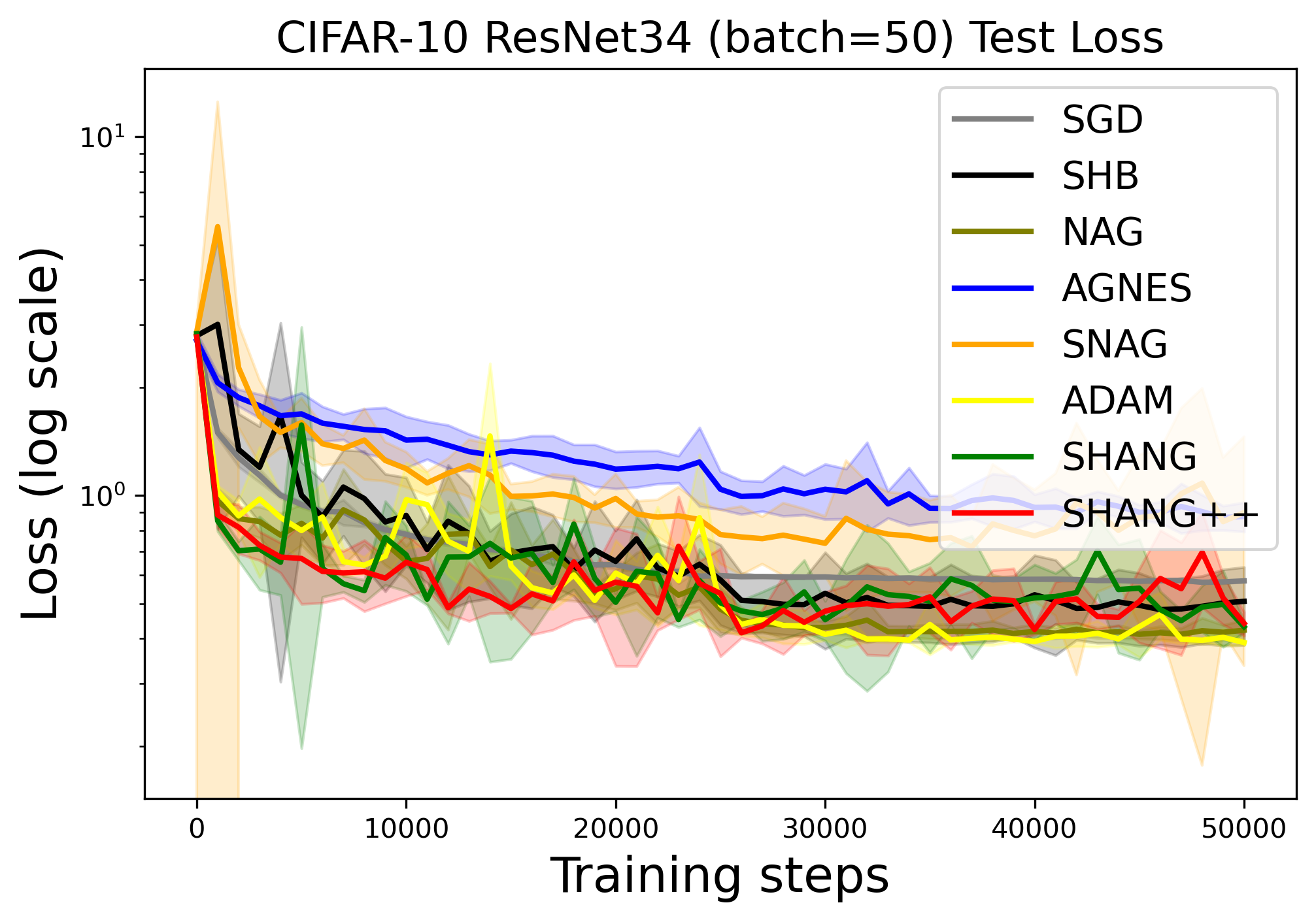}
			\end{subfigure}\hfill
			\begin{subfigure}{0.45\textwidth}
				\centering
				\includegraphics[width=\linewidth]{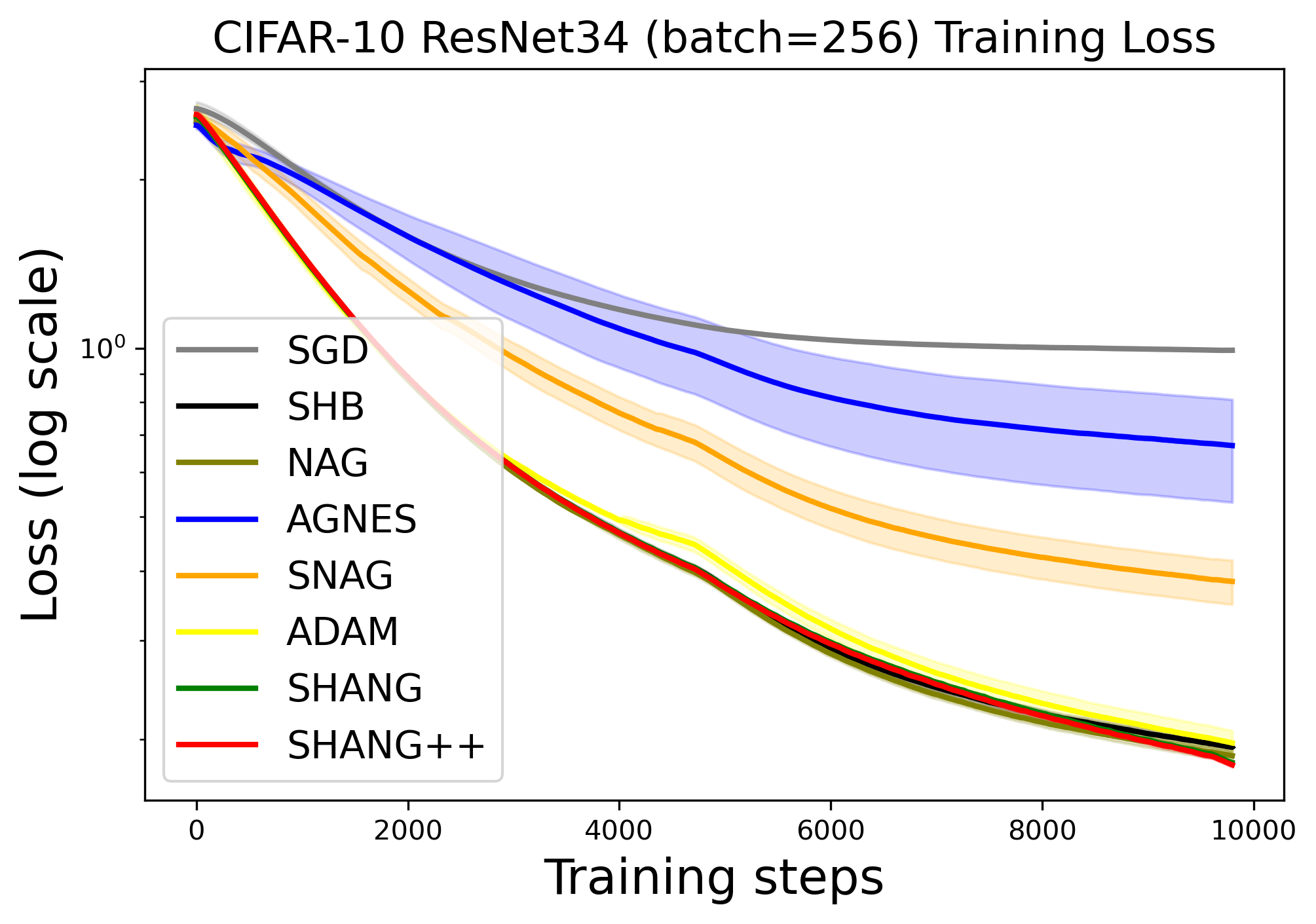}
			\end{subfigure}\hfill
			\begin{subfigure}{0.45\textwidth}
				\centering
				\includegraphics[width=\linewidth]{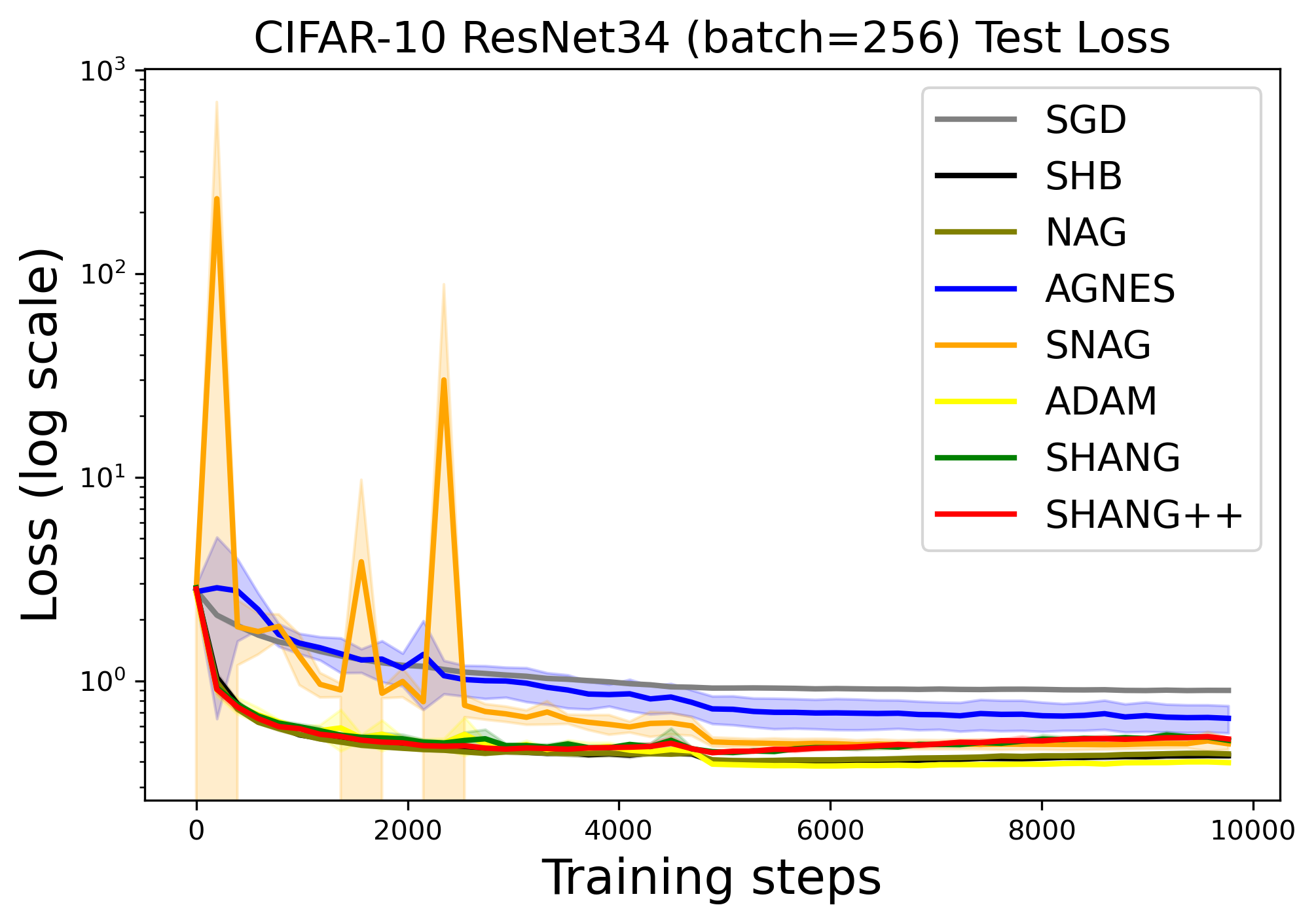}
			\end{subfigure}
			\caption{\small Training loss (left) and test loss (right) in log scale (running average with decay $0.99$) on CIFAR-10 with ResNet-34, for batch sizes $32$ (top row), $50$ (middle row), and $256$ (bottom row).}	
			\label{ResNet} 
		\end{figure}
		
		For hyperparameter selection, our two methods were evaluated under three settings: $\alpha = 0.5$ with $\gamma \in \{1, 1.5, 2\}$ for LeNet-5, $\gamma \in \{5, 10\}$ for ResNet-34, and $\gamma \in \{10, 15\}$ for ResNet-50. For SHANG++, we fixed $m = 1.5$. AGNES used the default parameters recommended by its authors, $(\eta, \alpha, m) = (0.01, 0.001, 0.99)$, which have shown strong performance across tasks. For SNAG, we used the two-parameter variant $(\eta,\beta)$ proposed by the original authors for machine-learning tasks. Hyperparameters were selected by grid search over $\eta \in \{0.5, 0.1, 0.05, 0.01, 0.005, 0.001\}$ and $\beta \in \{0.7, 0.8, 0.9, 0.99\}$. The best choice was $(\eta,\beta)=(0.05,0.9)$, which matches the recommended choice for training CNNs on CIFAR. For SGD, SHB, and NAG, we used a fixed learning rate $\eta=0.001$; for SHB and NAG, the momentum coefficient was $0.99$.
		
		For Adam~\cite{KingmaBa2015}, we used the default parameters $(\eta,\beta_1,\beta_2)=(0.001,0.9,0.999)$. Adam is included as a strong adaptive reference baseline. It uses coordinatewise adaptive learning rates based on first- and second-moment estimates, whereas the other methods are non-adaptive in this coordinatewise sense.
		
		Figures~\ref{MNIST-LeNet5}, \ref{ResNet}, and \ref{CIFAR100} report the results of SHANG and SHANG++ with $\alpha=0.5$, $m=1.5$, and $\gamma=2,5,10$, respectively. Figure~\ref{ResNet} shows the training and test losses of ResNet-34 on CIFAR-10 under different batch sizes, with all algorithmic hyperparameters fixed across batch sizes. Batch size strongly affects gradient variance: smaller batches increase noise, while larger batches reduce it. At batch size $256$, all methods are stable and the gaps narrow. At batch size $50$, NAG, SNAG, and AGNES oscillate with wider bands, and AGNES also plateaus higher. At batch size $32$, the differences among methods become clear.
		
		\begin{figure}[!htbp]
			\centering
			\begin{subfigure}{0.45\textwidth}
				\centering
				\includegraphics[width=\linewidth]{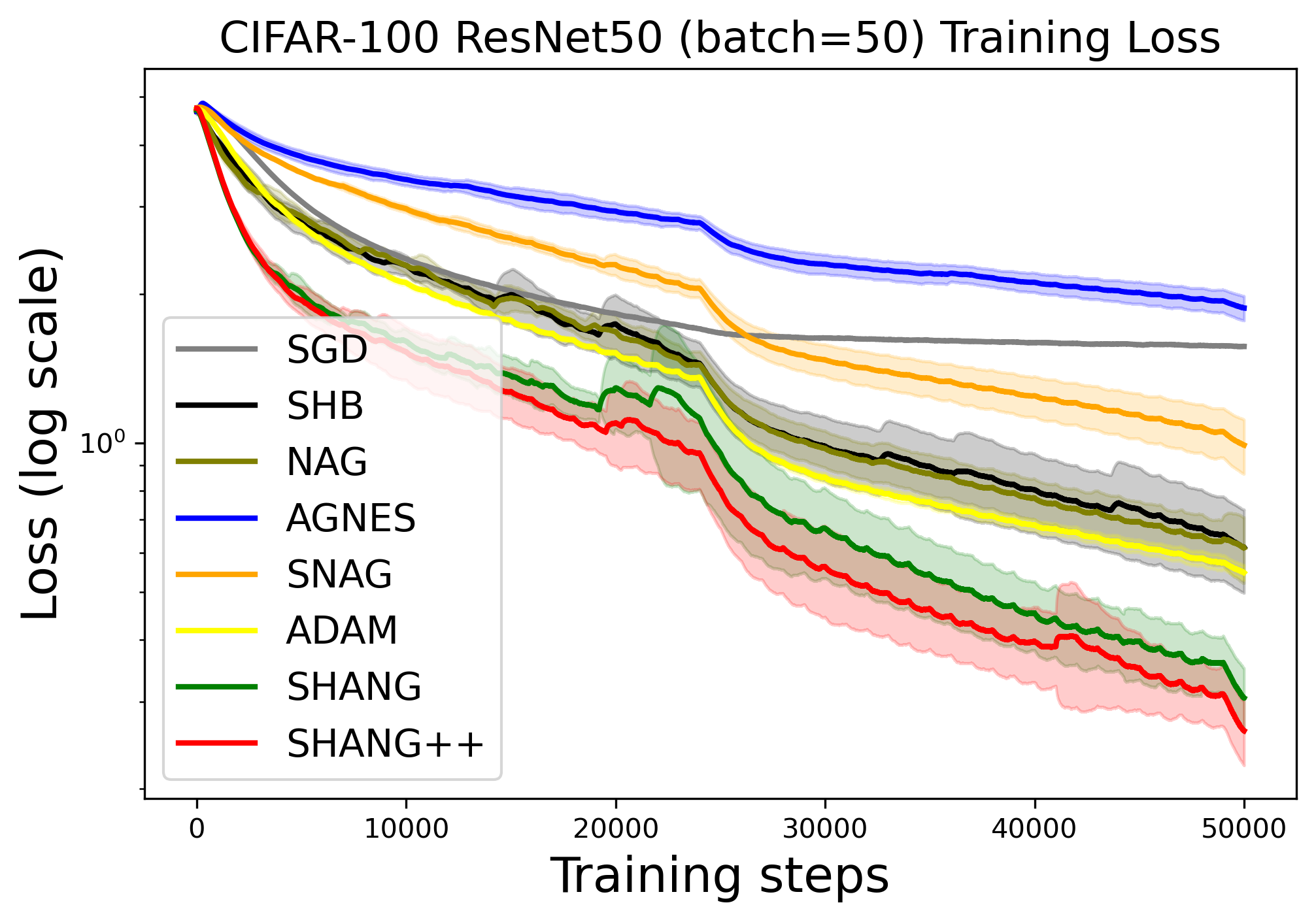}
			\end{subfigure}\hfill
			\begin{subfigure}{0.45\textwidth}
				\centering
				\includegraphics[width=\linewidth]{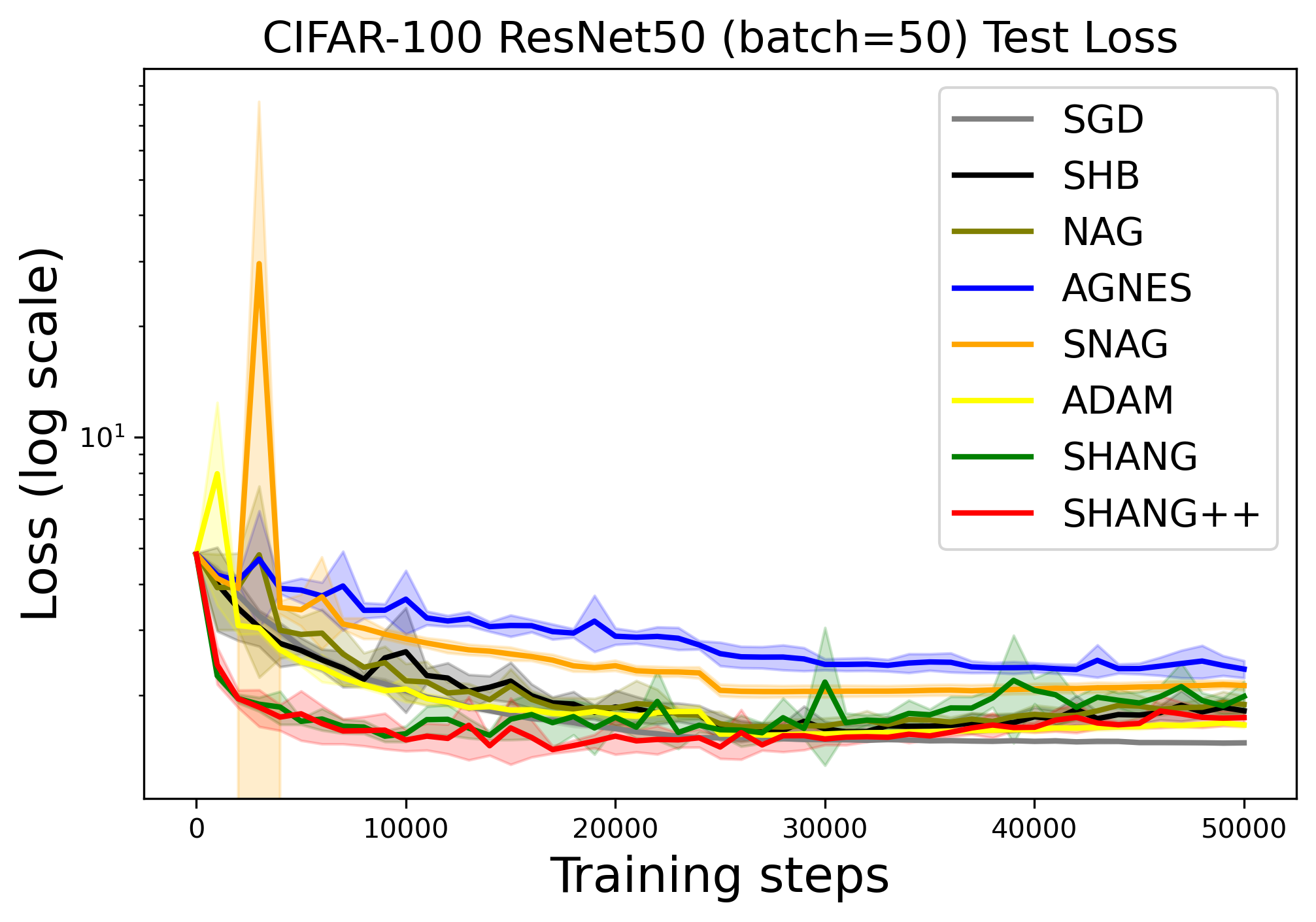}
			\end{subfigure}\hfill
			\caption{\small Training and test loss (log scale, running average with decay $0.99$) on CIFAR-100 with ResNet-50 (batch size $50$).}
			\label{CIFAR100}
		\end{figure}
		
		In these small-batch regimes, SHANG and SHANG++ consistently outperform the other first-order stochastic momentum methods. When the batch size falls below $50$, AGNES and SNAG lose their acceleration advantage over SGD, whereas SHANG, SHANG++, and Adam still give clear improvements. As also observed in~\cite{HermantEtAl2025}, non-variance-reduced accelerated methods often lose acceleration at very small batch sizes. SHANG and SHANG++ appear to remain robust at smaller batch-size thresholds.
		
		\begin{table}[htbp]       
			\centering            
			\caption{Test accuracy of SGD, SHB, NAG, Adam, AGNES, SNAG, SHANG, and SHANG++ on MNIST (LeNet-5), CIFAR-10 (ResNet-34), and CIFAR-100 (ResNet-50). Here $b$ is batch size. Larger values are better. Best is in bold and second best is underlined.}
			\label{test accuracy}      
			\renewcommand{\arraystretch}{1.15}	
			\resizebox{0.92\textwidth}{!}{    
				{\large 
					\begin{tabular}{ccccccccc}     
						\toprule             
						& SGD & SHB  & NAG & Adam & AGNES & SNAG & SHANG & SHANG++   \\
						\midrule
						LeNet-5 
						& $91.07$ & $98.98$  &$98.9$& $\underline{99.07}$ & $98.88$ & $\underline{99.07}$ & $99.06$ & $\mathbf{99.11}$  \\
						$(b=50)$ 
						& $\pm 0.11$&$\pm0.05$  & $\pm 0.08$  & $\pm 0.07$ & $\pm 0.09$  & $\pm 0.08$ & $\pm 0.02$ & $\pm 0.03$  \\
						ResNet-34 
						& $81.74$ & $78.9$ & $81.28$ & $\mathbf{86.99}$ & $67.45$  & $75.58$ & $84.5$ & $\underline{85.36}$ \\
						$(b=32)$ 
						& $\pm 0.38$ & $\pm 1.67$ & $\pm 1.58$& $\pm 0.14$  & $\pm 7.7$  & $\pm 6.02$ & $\pm 2$  & $\pm 1.42$  \\
						ResNet-34 
						& $79.91$ & $84.59$ & $86.43$ & $\underline{87.38}$ & $70.49$ & $77.65$ & $87.15$ & $\mathbf{87.4}$ \\
						$(b=50)$ 
						& $\pm 0.11$ & $\pm 2.62$ & $\pm 0.81$ & $\pm 0.26$ & $\pm 2.51$  & $\pm 2.7$ & $\pm 0.82$ & $\pm 0.5$ \\
						ResNet-34 
						& $68.49$ & $87.6$ & $\underline{87.61}$ & $\mathbf{88.23}$ & $77.84$ & $84.5$ & $86.67$ & $86.57$ \\
						$(b=256)$
						& $\pm 0.19$ & $\pm 0.27$ & $\pm 0.29$ & $\pm 0.11$ & $\pm 3.7$  & $\pm 0.92$ & $\pm 0.13$ & $\pm 0.17$ \\
						ResNet-50 
						& $58.31$ & $58.17$ & $57.66$ & $59.87$ & $42.82$ & $49.51$ & $\underline{63.31}$ & $\mathbf{65.02}$\\
						$(b=50)$ 
						& $\pm 0.51$ & $\pm 1.99$ & $\pm 1.44$ & $\pm 0.61$ & $\pm 1.24$ & $\pm 1.56$  & $\pm 0.93$ & $\pm 1.25$\\
						\bottomrule
					\end{tabular}
			}}
		\end{table}
		
		Figure~\ref{CIFAR100} shows ResNet-50 training and test losses on CIFAR-100. SHANG and SHANG++ are competitive with, or better than, the non-adaptive baselines and remain competitive with Adam. Table~\ref{test accuracy} summarizes the mean final test accuracy over five independent runs. SHANG and SHANG++ are comparable to Adam, often surpass AGNES and SNAG, and clearly improve over SGD and NAG. The absolute accuracies are lower than standard benchmarks because we use small batch sizes and only $50$ training epochs to stress-test optimizer stability, rather than to reach full convergence.
		
		An interesting observation is that SGD attains the lowest test loss, but not the best classification accuracy; see Table~\ref{test accuracy}. This mismatch is consistent with prior findings: SGD on cross-entropy with hard labels can lead to confidently wrong predictions~\cite{thulasidasan2020mixuptrainingimprovedcalibration}.
		
		\subsection{Robustness to Multiplicative Gradient Noise}\label{testsigma}
		Our theory predicts that the time-scale coupling $(\alpha,\gamma)$ in SHANG and $(\alpha,\gamma,m)$ in SHANG++ mitigates multiplicative gradient noise. To test this, we fix one hyperparameter configuration per optimizer and evaluate $\sigma \in \{0,0.05,0.1,0.2,0.5\}$.
		
		We test robustness by adding multiplicative noise to the mini-batch gradient. If $g(x_k)$ is the mini-batch gradient at iteration $k$, we use
			$$
			g^\sigma(x_k)
			=
			\bigl(1+\sigma Z_k\bigr) g(x_k),
			\qquad Z_k\sim \mathcal N(0,I_d),
			$$
			with componentwise multiplication. Thus $\sigma$ controls the strength of the added perturbation. In the full-gradient case $g(x_k)=\nabla f(x_k)$, this gives the MNS model in~\eqref{MNS}. In the mini-batch setting, the perturbation is added on top of the intrinsic sampling noise.
		
		This one-shot protocol tests each optimizer's robustness without re-tuning. All experiments use CIFAR-10 with ResNet-34 and batch size $50$, with the same settings as in Subsection~\ref{classification}. We train for $100$ epochs and average over three seeds. The final validation error at epoch $100$ is reported.
		
		\smallskip
		
		\noindent
		\begin{minipage}{0.5125\linewidth}
			\captionsetup{width=1\textwidth}
			\captionof{table}{Relative change in final classification error compared with $\sigma=0$ (lower is better; negative values indicate improvement). Values are averaged over three seeds.}
			\renewcommand{\arraystretch}{1.2}
			\resizebox{6.5cm}{!}{
				\begin{tabular}{@{}lrrrr@{}}
					\toprule
					\multirow{2}{*}{Method} &
					\multicolumn{4}{c}{Relative degradation $\Delta(\%)$ at $\sigma$} \\
					\cmidrule(lr){2-5}
					& 0.05 & 0.1 & 0.2 & 0.5 \\
					\midrule
					SHANG   & $-2.5$ & $-2.1$ & $-1.0$ & $-0.2$ \\
					SHANG++ & $+3.4$ & $-0.6$ & $-2.1$ & $-0.9$ \\
					AGNES   & $-14.4$ & $+16.0$ & $+14.6$ & $+13.5$ \\
					SNAG    & $-2.0$ & $-2.1$ & $-5.0$ & $0.7$ \\
					\bottomrule
				\end{tabular}\\ \\
			}
			\label{tab:sigma_delta}
			\vspace{26pt}
		\end{minipage}
		\qquad
		\begin{minipage}{0.43\linewidth}
			\captionsetup{width=0.99\textwidth}
			\captionof{figure}{\small Validation error under varying multiplicative noise level $\sigma$. Lower is better.}
			\includegraphics[width=\linewidth]{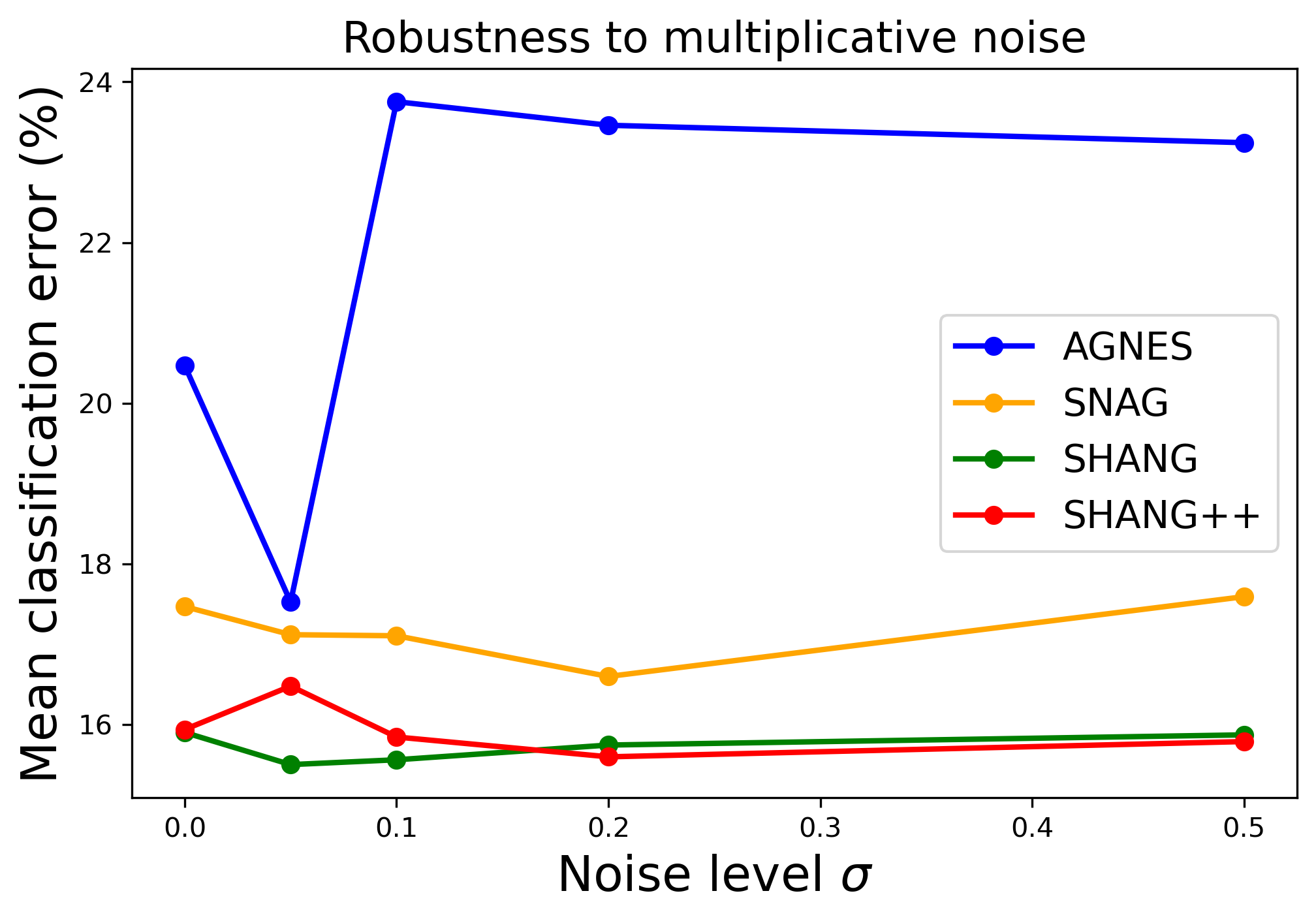}
			\label{fig:sigma_curve}
		\end{minipage}
		
		Figure~\ref{fig:sigma_curve} shows the mean final classification error under varying noise levels, and Table~\ref{tab:sigma_delta} reports the relative degradation
		$$
		\Delta(\sigma)
		=
		100\cdot
		\frac{\mathbb{E}(\sigma)-\mathbb{E}(0)}{\mathbb{E}(0)}.
		$$
		Here $\mathbb{E}(\sigma)$ denotes the mean classification error, averaged over three seeds, at noise level $\sigma$. Smaller $\Delta(\sigma)$ indicates better robustness.\begin{enumerate}
			\item At $\sigma=0$, SHANG and SHANG++ reach $15.9\%$, outperforming SNAG ($17.5\%$) and AGNES ($20.5\%$).
			\item At $\sigma=0.1$, SHANG slightly improves to $15.6\%$, SHANG++ remains stable at $15.9\%$, SNAG improves slightly to $17.1\%$, and AGNES degrades to $23.8\%$.
			\item At $\sigma=0.5$, SHANG and SHANG++ remain near $16\%$, while SNAG rises to $17.6\%$ and AGNES drifts to $23.2\%$, about a $13.5\%$ relative increase.
		\end{enumerate}
		
		These results are consistent with the Lyapunov analysis: the time-scale coupling $(\alpha,\gamma,m)$ suppresses the $\sigma^2$ amplification and gives stable performance without re-tuning. SNAG is stable but less accurate, while AGNES is more sensitive to noise.
		
		\subsection{Image Reconstruction with Small Batch Size}
		We further test the algorithms on image reconstruction with small-batch training, which introduces large gradient noise. We train a lightweight U-Net~\citep{RonnebergerFischerBrox2015}, with base channels $32 \to 64 \to 128$, bilinear up-sampling, and feature concatenation, on CIFAR-10 using batch size $5$. We compare SHANG $(\alpha=0.5, \gamma=0.5)$ and SHANG++ $(\alpha=0.5, \gamma=0.5, m=1)$ with SNAG, AGNES, NAG, SGD, SHB, and Adam. All other settings follow the earlier experiments.
		
		\begin{figure}[!htbp]
			\centering
			\begin{subfigure}{0.45\textwidth}
				\centering
				\includegraphics[width=\linewidth]{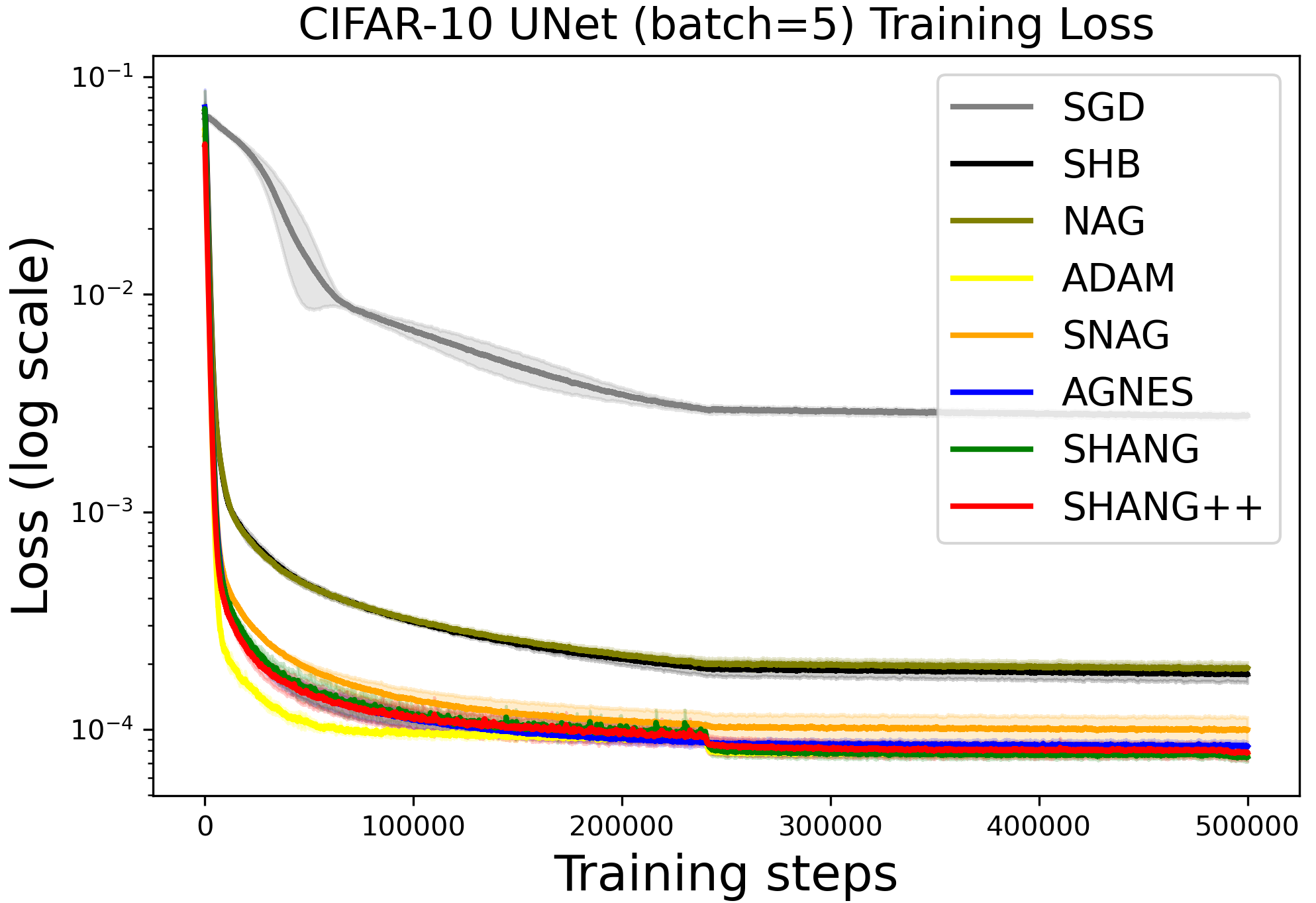}
			\end{subfigure}\hfill
			\begin{subfigure}{0.45\textwidth}
				\centering
				\includegraphics[width=\linewidth]{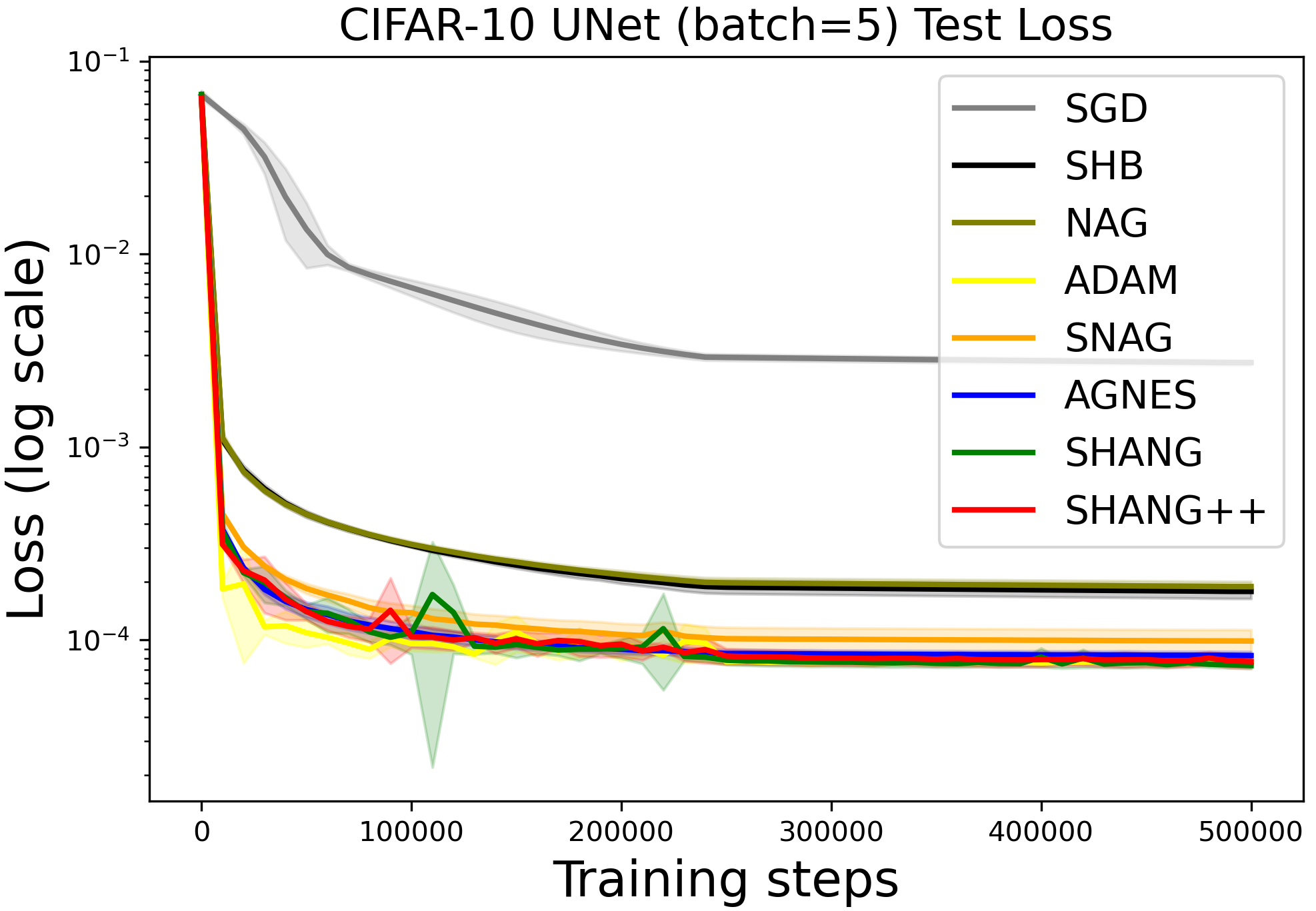}
			\end{subfigure}\hfill
			\caption{\small Training and test loss (log scale, running average with decay $0.99$) on CIFAR-10 using U-Net with batch size $5$.}	
			\label{UNet}
		\end{figure}
		Figure~\ref{UNet} shows training and test losses. Adam achieves the lowest loss because of its adaptive learning rate, while SHANG and SHANG++ outperform the other non-adaptive methods. In particular, SHANG++ trains stably in this high-noise regime, supporting its practical robustness.
		
		\section{Convergence Analysis}\label{analysis}
		Throughout this section, we work on a complete probability space
		\((\Omega,\mathcal F,\mathbb P)\) equipped with a filtration
		\(\{\mathcal F_k\}_{k\ge0}\). We define \(\mathcal F_k\) as the information available before sampling the stochastic gradient at \(x_k\); equivalently, \(\mathcal F_k\) contains the initial data and all stochastic-gradient evaluations up to \(g(x_{k-1})\). Thus \(x_k\) is \(\mathcal F_k\)-measurable, while \(g(x_k)\) is sampled independently afterward. In the one-step analysis, after \(g(x_k)\) is sampled, the quantities \(x_k^+\) and \(x_{k+1}\) are fixed before \(g(x_{k+1})\) is sampled. With our indexing, this intermediate information is represented by \(\mathcal F_{k+1}\), namely the previous information augmented by the stochastic-gradient evaluation \(g(x_k)\).
		
		With this convention in place, we now prove the main convergence results.
		
		\subsection{Proof of Theorem \ref{theorem1}}\label{shanganalysis}
		Before presenting the proof of Theorem \ref{theorem1}, we first establish several auxiliary lemmas, beginning with one that relies on conditional expectations under the MNS assumption.
		\begin{lemma}\label{lemma1}
			Suppose \(x_k\) is generated by SHANG/SHANG++, and \(g(x_k)\) is a conditionally unbiased stochastic gradient estimator satisfying the MNS condition with respect to \(\mathcal F_k\). Then the following statements hold:
			
			1. $\mathbb{E} \left[g(x_k) \mid \mathcal{F}_k \right] =  \nabla f(x_k)$
			
			2. $\mathbb{E} \left[ \| g(x_k) - \nabla f(x_k) \|^2 \mid \mathcal{F}_k \right] \le \sigma^2 \| \nabla f(x_k) \|^2$
			
			3. $\mathbb{E} \left[\langle g(x_k), \nabla f(x_k) \rangle \mid \mathcal{F}_k  \right] = \| \nabla f(x_k) \|^2 $
			
			4. $\mathbb{E} \left[\| g(x_k) \|^2 \mid \mathcal{F}_k  \right] \le (1+\sigma^2) \| \nabla f(x_k) \|^2$
		\end{lemma}
		\begin{proof}
			The first statement is the conditional unbiasedness assumption, and the second statement is the MNS assumption. Since \(x_k\) is \(\mathcal F_k\)-measurable, \(\nabla f(x_k)\) is also \(\mathcal F_k\)-measurable. Then, by Theorem 8.14 in \cite{Klenke2013}, we have 
			$$
			\begin{aligned}
				\mathbb{E} \left[\langle g(x_k), \nabla f(x_k) \rangle \mid \mathcal{F}_k \right]  
				= \langle \mathbb{E} \left[ g(x_k)\mid \mathcal{F}_k \right], \nabla f(x_k) \rangle =  \| \nabla f(x_k) \|^2 
			\end{aligned}
			$$
			For the fourth result, using the previous results, we have
			$$\begin{aligned}
				\mathbb{E} \left[\| g(x_k) \|^2 \mid \mathcal{F}_k \right] &= \mathbb{E} \left[\| g(x_k) - \nabla f(x_k) \|^2 + 2 \langle g(x_k), \nabla f(x_k) \rangle - \| \nabla f(x_k) \|^2  \mid \mathcal{F}_k \right]  \\
				& = \mathbb{E} \left[\| g(x_k) - \nabla f(x_k) \|^2 \mid \mathcal{F}_k \right] +  \| \nabla f(x_k) \|^2 \\
				& \le  (1 + \sigma^2) \| \nabla f(x_k) \|^2 \\
			\end{aligned}$$
		\end{proof}
		
		To estimate the one-step Lyapunov difference conditionally, insert the intermediate term $\mathcal{E}(z_{k+1};\gamma_{k+1})$ and write
		\begin{equation}\label{equ151}
			\begin{aligned}
				&\mathcal{E}(z_{k+1}^+;\gamma_{k+1})-\mathcal{E}(z_k^+;\gamma_k) \\
				={}&
				\mathcal{E}(z_{k+1}^+;\gamma_{k+1})-\mathcal{E}(z_{k+1};\gamma_{k+1})
				+
				\mathcal{E}(z_{k+1};\gamma_{k+1})-\mathcal{E}(z_k^+;\gamma_k).
			\end{aligned}
		\end{equation}
		Under the parameter choices in Theorem~\ref{theorem1}, $\alpha_k\beta_k=\alpha_{k+1}\beta_{k+1}\le 1/(L(1+\sigma^2))$. Applying Lemma~\ref{lemma2} conditionally gives
		\begin{equation}\label{equ153}
			\begin{aligned}
				&\mathbb{E}\left[\mathcal{E}(z_{k+1}^+;\gamma_{k+1})-\mathcal{E}(z_{k+1};\gamma_{k+1})\mid\mathcal F_{k+1}\right] \\
				={}& \mathbb{E}\left[f(x_{k+1}^+)-f(x_{k+1})\mid\mathcal F_{k+1}\right] \le -\frac{\alpha_k\beta_k}{2}\|\nabla f(x_{k+1})\|^2 .
			\end{aligned}
		\end{equation}
		Thus the shift from $x_{k+1}$ to $x_{k+1}^+$ gives the needed descent.
		
		The next lemma bounds the second term.
		
		\begin{lemma}\label{lemma4}
			Let $f\in\mathcal{S}_{\mu,L}$. Starting from $x_0=v_0$, generate $(x_k,v_k)$ by~\eqref{SHANG}. Let the stochastic gradients be conditionally unbiased and satisfy the MNS condition~\eqref{eq:mns}. Define $x_k^+$ by~\eqref{auxi1}. Let $\mathcal F_{k+1}$ denote the information available before sampling $g(x_{k+1})$, so that $x_k^+$, $x_{k+1}$, and $v_k$ are $\mathcal F_{k+1}$-measurable. Then
			\begin{equation}\label{lemma4eq}
				\begin{aligned}
					&\mathbb{E}\left[\mathcal{E}(z_{k+1};\gamma_{k+1})-\mathcal{E}(z_k^+;\gamma_k)\mid\mathcal F_{k+1}\right] \\
					&\le -\alpha_k\mathbb{E}\left[\mathcal{E}(z_{k+1};\gamma_{k+1})\mid\mathcal F_{k+1}\right]
					+\frac{\alpha_k^2(1+\sigma^2)}{2\gamma_k}\|\nabla f(x_{k+1})\|^2 .
				\end{aligned}
			\end{equation}
		\end{lemma}
		
		\begin{proof}
			Using the update rules~\eqref{SHANG}, we expand
			\begin{equation}\label{equ9}
				\begin{aligned}
					&\mathcal{E}(z_{k+1};\gamma_{k+1})-\mathcal{E}(z_k^+;\gamma_k) \\
					\le{}& -\alpha_k\langle \nabla f(x_{k+1})-\nabla f(x^\star),x_{k+1}-x^\star\rangle
					+\alpha_k\langle g(x_{k+1}),v_k-v_{k+1}\rangle \\
					&\quad +\alpha_k\mu\langle v_{k+1}-x^\star,x_{k+1}-v_{k+1}\rangle
					+\alpha_k\langle \nabla f(x_{k+1})-g(x_{k+1}),v_k-x^\star\rangle \\
					&\quad +\frac{\alpha_k(\mu-\gamma_{k+1})}{2}\|v_{k+1}-x^\star\|^2
					-D_{\mathcal E}(z_k^+,z_{k+1};\gamma_k).
				\end{aligned}
			\end{equation}
			Taking conditional expectation with respect to $\mathcal F_{k+1}$, the martingale term vanishes:
			$$
			\mathbb{E}\left[\langle \nabla f(x_{k+1})-g(x_{k+1}),v_k-x^\star\rangle\mid\mathcal F_{k+1}\right]=0.
			$$
			By strong convexity,
			\begin{equation}\label{equ10}
				-\langle \nabla f(x_{k+1})-\nabla f(x^\star),x_{k+1}-x^\star\rangle
				\le
				-\bigl(f(x_{k+1})-f(x^\star)\bigr)-\frac{\mu}{2}\|x_{k+1}-x^\star\|^2.
			\end{equation}
			Also,
			\begin{equation}\label{equ11}
				\alpha_k\mu\langle v_{k+1}-x^\star,x_{k+1}-v_{k+1}\rangle
				=
				\frac{\alpha_k\mu}{2}\left(\|x_{k+1}-x^\star\|^2-\|x_{k+1}-v_{k+1}\|^2-\|v_{k+1}-x^\star\|^2\right).
			\end{equation}
			For the stochastic cross term, Young's inequality gives
			\begin{equation}\label{equ12}
				\begin{aligned}
					\alpha_k\langle g(x_{k+1}),v_k-v_{k+1}\rangle
					&\le
					\frac{\alpha_k^2}{2\gamma_k}\|g(x_{k+1})\|^2
					+\frac{\gamma_k}{2}\|v_k-v_{k+1}\|^2.
				\end{aligned}
			\end{equation}
			Using the conditional MNS bound,
			$$
			\mathbb{E}\left[\|g(x_{k+1})\|^2\mid\mathcal F_{k+1}\right]
			\le
			(1+\sigma^2)\|\nabla f(x_{k+1})\|^2.
			$$
			Finally,
			\begin{equation}\label{equ13}
				-D_{\mathcal E}(z_k^+,z_{k+1};\gamma_k)
				=
				-D_f(x_k^+,x_{k+1})-\frac{\gamma_k}{2}\|v_k-v_{k+1}\|^2
				\le
				-\frac{\gamma_k}{2}\|v_k-v_{k+1}\|^2 .
			\end{equation}
			Combining \eqref{equ9}--\eqref{equ13} gives
			\begin{equation}\label{equ14}
				\begin{aligned}
					&\mathbb{E}\left[\mathcal{E}(z_{k+1};\gamma_{k+1})-\mathcal{E}(z_k^+;\gamma_k)\mid\mathcal F_{k+1}\right] \\
					\le{}&
					-\alpha_k\mathbb{E}\left[\mathcal{E}(z_{k+1};\gamma_{k+1})\mid\mathcal F_{k+1}\right]
					+\frac{\alpha_k^2(1+\sigma^2)}{2\gamma_k}\|\nabla f(x_{k+1})\|^2,
				\end{aligned}
			\end{equation}
			where the extra negative term $-\frac{\alpha_k\mu}{2}\|x_{k+1}-v_{k+1}\|^2$ is dropped. 
			This proves the lemma.
		\end{proof}
		
		Lemma \ref{lemma4} provides an upper bound for the second term in~\eqref{equ151}. Combining the two bounds, we obtain:
		\begin{equation}\label{equ152}
			\begin{aligned}
				&\mathbb{E}\left[\mathcal{E}(z_{k+1}^+; \gamma_{k+1})- \mathcal{E}(z_k^+;\gamma_{k})\right]   \\
				\le{}& \mathbb{E}\left[- \alpha_k \mathcal{E}(z_{k+1}; \gamma_{k+1}) +( \frac{\alpha_k^2(1+\sigma^2)}{2\gamma_{k}} -\frac{\alpha_{k}\beta_k}{2} ) \| \nabla f(x_{k+1}) \|^2 \right]
			\end{aligned}
		\end{equation}
		Indeed, by the definition of \(\mathcal{E}\), we have
		\[
		\begin{aligned}
			\mathbb{E}\left[-\alpha_k \mathcal{E}(z_{k+1};\gamma_{k+1})\right]
			={}&
			\mathbb{E}\left[
			-\alpha_k\bigl(f(x_{k+1})-f(x^\star)\bigr)
			-\frac{\alpha_k\gamma_{k+1}}{2}\|v_{k+1}-x^\star\|^2
			\right].
		\end{aligned}
		\]
		Applying Lemma~\ref{lemma2} to the update from \(x_{k+1}\) to \(x_{k+1}^+\) gives
		\[
		\mathbb{E}\left[f(x_{k+1}^+)\right]
		\le
		\mathbb{E}\left[f(x_{k+1})
		-\frac{\alpha_k\beta_k}{2}\|\nabla f(x_{k+1})\|^2
		\right].
		\]
		Therefore,
		\[
		\begin{aligned}
			\mathbb{E}\left[-\alpha_k \mathcal{E}(z_{k+1};\gamma_{k+1})\right]
			&\le
			\mathbb{E}\left[
			-\alpha_k\bigl(f(x_{k+1}^+)-f(x^\star)\bigr)
			-\frac{\alpha_k\gamma_{k+1}}{2}\|v_{k+1}-x^\star\|^2
			\right.\\
			&\qquad\left.
			-\frac{\alpha_k^2\beta_k}{2}\|\nabla f(x_{k+1})\|^2
			\right] \\
			&=
			\mathbb{E}\left[
			-\alpha_k\mathcal{E}(z_{k+1}^+;\gamma_{k+1})
			-\frac{\alpha_k^2\beta_k}{2}\|\nabla f(x_{k+1})\|^2
			\right].
		\end{aligned}
		\]
		Then
		\begin{equation}
			\begin{aligned}
				&\mathbb{E}\left[\mathcal{E}(z_{k+1}^+; \gamma_{k+1})- \mathcal{E}(z_k^+;\gamma_{k})\right]   \\
				\le{}& \mathbb{E}\left[- \alpha_k \mathcal{E}(z_{k+1}^+; \gamma_{k+1}) +( \frac{\alpha_k^2(1+\sigma^2)}{2\gamma_{k}} -\frac{\alpha_{k}\beta_k(1+\alpha_{k})}{2} ) \| \nabla f(x_{k+1}) \|^2 \right]
			\end{aligned}
		\end{equation}
		By the choice $\beta_k=\alpha_k(1+\sigma^2)/\gamma_k$, we have
		$$\frac{\alpha_k^2(1+\sigma^2)}{2\gamma_{k}} -\frac{\alpha_{k}\beta_k(1+\alpha_{k})}{2}=- \frac{\alpha_k^3(1+\sigma^2)}{2\gamma_{k}} \le 0.$$
		Then the coefficient of the gradient term is nonpositive. Thus,
		\begin{align*}
			\mathbb{E}\left[\mathcal{E}(z_{k+1}^+; \gamma_{k+1})- \mathcal{E}(z_k^+;\gamma_{k})\right]  \le  \mathbb{E}\left[- \alpha_k \mathcal{E}(z_{k+1}^+; \gamma_{k+1}) \right]
		\end{align*}
		
		For the choices of $\alpha_k$ and $\gamma_k$ in Theorem~\ref{theorem1},
		we can verify that
		\[
		\frac{\gamma_{k+1}-\gamma_k}{\alpha_k}
		\le \mu-\gamma_{k+1},
		\]
		and hence the condition~\eqref{SHANG:c} is satisfied. We now apply the preceding Lyapunov estimate to the two cases in
		Theorem~\ref{theorem1}.
		
		\noindent (1). When $ \mu > 0$, set $\gamma = \mu$ and $\alpha \le \frac{1}{1+\sigma^2}\sqrt{\frac{\mu}{L}}$. Then,
		\begin{equation}
			\mathbb{E}\left[\mathcal{E}(z_{k+1}^+; \mu)\right] \le (1+\alpha)^{-1} \mathbb{E}\left[\mathcal{E}(z_{k}^+; \mu)\right]  \le (1+\alpha)^{-(k+1)} \mathbb{E} [\mathcal{E}(z_{0}^+; \mu)]
		\end{equation}
		
		\noindent (2). When $ \mu = 0$, set $\alpha_k = \frac{2}{k+1}$, $\beta_{k} = \frac{\alpha_k(1+\sigma^2)}{\gamma_{k}}$ and $\gamma_k = \alpha_k^2(1+\sigma^2)^2L$. Then
		\begin{equation}
			\mathbb{E}\left[\mathcal{E}(z_{k+1}^+; \gamma_{k+1})\right] \le \frac{1}{1+\alpha_k} \mathbb{E}\left[\mathcal{E}(z_{k}^+; \gamma_k)\right]  \le \prod_{i=0}^k \frac{1}{1+\alpha_i} \mathbb{E} [ \mathcal{E}(z_{0}^+; \gamma_0)]
		\end{equation}
		
		This proves Theorem~\ref{theorem1}.
		
		\subsection{Proof of Theorem \ref{theorem2}}\label{shang++scanalysis}
		For SHANG++ in the strongly convex setting~\eqref{equ26}, we use the same intermediate-step decomposition. To estimate the one-step Lyapunov difference, we insert $\mathcal{E}(z_{k+1};\mu)$ and write
		\begin{equation}\label{equ154}
			\begin{aligned}
				&\mathcal{E}(z_{k+1}^+;\mu)-\mathcal{E}(z_k^+;\mu) \\
				={}&
				\mathcal{E}(z_{k+1}^+;\mu)-\mathcal{E}(z_{k+1};\mu)
				+
				\mathcal{E}(z_{k+1};\mu)-\mathcal{E}(z_k^+;\mu).
			\end{aligned}
		\end{equation}
		Under the parameter choices in Theorem~\ref{theorem2},
		$$
		\tilde{\alpha}\beta
		=
		\frac{\tilde{\alpha}^2(1+\sigma^2)}{\mu}
		\le
		\frac{1}{L(1+\sigma^2)}.
		$$
		Applying the conditional descent estimate in Lemma~\ref{lemma2} to the auxiliary step
		$$
		x_{k+1}^+=x_{k+1}-\tilde{\alpha}\beta g(x_{k+1}),
		$$
		with $\alpha_{k+1}\beta_{k+1}$ replaced by $\tilde{\alpha}\beta$, yields
		\begin{equation}\label{equ154_}
			\begin{aligned}
				&\mathbb{E}\left[\mathcal{E}(z_{k+1}^+;\mu)-\mathcal{E}(z_{k+1};\mu)\mid\mathcal F_{k+1}\right] \\
				={}&
				\mathbb{E}\left[f(x_{k+1}^+)-f(x_{k+1})\mid\mathcal F_{k+1}\right]  
				\le
				-\frac{\tilde{\alpha}^2(1+\sigma^2)}{2\mu}\|\nabla f(x_{k+1})\|^2 .
			\end{aligned}
		\end{equation}
		The next lemma controls the conditional difference
		$$
		\mathbb{E}\left[\mathcal{E}(z_{k+1};\mu)-\mathcal{E}(z_k^+;\mu)\mid\mathcal F_{k+1}\right].
		$$
		
		\begin{lemma}\label{lemma7}
			Let $f\in\mathcal{S}_{\mu,L}$. Starting from $x_0=v_0$, generate $(x_k,v_k)$ by~\eqref{equ26}. Let the stochastic gradients be conditionally unbiased and satisfy the MNS condition~\eqref{eq:mns}. Define $x_k^+$ by~\eqref{auxi2}. Let $\mathcal F_{k+1}$ be the information available before sampling $g(x_{k+1})$, so that $x_k^+$, $x_{k+1}$, and $v_k$ are $\mathcal F_{k+1}$-measurable. Then
			\begin{equation}\label{lemma7eq}
				\begin{aligned}
					&\mathbb{E}\left[\mathcal{E}(z_{k+1};\mu)-\mathcal{E}(z_k^+;\mu)\mid\mathcal F_{k+1}\right] \\
					\le{}&
					-\alpha\mathbb{E}\left[\mathcal{E}(z_{k+1};\mu)\mid\mathcal F_{k+1}\right]
					+
					\frac{\tilde{\alpha}^2(1+\alpha)(1+\sigma^2)}{2\mu}
					\|\nabla f(x_{k+1})\|^2 .
				\end{aligned}
			\end{equation}
		\end{lemma}
		
		\begin{proof}
			Using the update rules~\eqref{equ26}, we first expand
			\begin{equation}\label{equ133}
				\begin{aligned}
					&\mathcal{E}(z_{k+1};\mu)-\mathcal{E}(z_k^+;\mu) \\
					={}&
					\tilde{\alpha}\langle \nabla f(x_{k+1})-\nabla f(x^\star),v_k-x_{k+1}\rangle
					-\alpha\langle g(x_{k+1}),v_{k+1}-x^\star\rangle\\
					&
					+\alpha\mu\langle v_{k+1}-x^\star,x_{k+1}-v_{k+1}\rangle 
					-D_{\mathcal E}(z_k^+,z_{k+1};\mu).
				\end{aligned}
			\end{equation}
			Split the gradient terms as
			\begin{equation}\label{equ134split}
				\begin{aligned}
					&\tilde{\alpha}\langle \nabla f(x_{k+1})-\nabla f(x^\star),v_k-x_{k+1}\rangle
					-\alpha\langle g(x_{k+1}),v_{k+1}-x^\star\rangle \\
					={}						&
					-\alpha\langle \nabla f(x_{k+1})-\nabla f(x^\star),x_{k+1}-x^\star\rangle \\
					&+
					\tilde{\alpha}\langle g(x_{k+1}),v_k-v_{k+1}\rangle
					+\alpha\tilde{\alpha}\langle g(x_{k+1}),x_{k+1}-v_{k+1}\rangle \\
					&
					+\tilde{\alpha}\langle \nabla f(x_{k+1})-g(x_{k+1}),v_k-x_{k+1}\rangle
					+\alpha\langle \nabla f(x_{k+1})-g(x_{k+1}),x_{k+1}-x^\star\rangle,
				\end{aligned}
			\end{equation}
			where we used $\alpha-\tilde{\alpha}=\alpha\tilde{\alpha}$. Taking conditional expectation with respect to $\mathcal F_{k+1}$, the last two terms in~\eqref{equ134split} vanish by conditional unbiasedness. By Young's inequality,
			\begin{equation}\label{equ137}
				\tilde{\alpha}\langle g(x_{k+1}),v_k-v_{k+1}\rangle
				\le
				\frac{\tilde{\alpha}^2}{2\mu}\|g(x_{k+1})\|^2
				+
				\frac{\mu}{2}\|v_k-v_{k+1}\|^2,
			\end{equation}
			and
			\begin{equation}\label{equ135}
				\alpha\tilde{\alpha}\langle g(x_{k+1}),x_{k+1}-v_{k+1}\rangle
				\le
				\frac{\alpha\tilde{\alpha}^2}{2\mu}\|g(x_{k+1})\|^2
				+
				\frac{\alpha\mu}{2}\|x_{k+1}-v_{k+1}\|^2.
			\end{equation}
			
			We expand the cross term using the square identity
			$$
			\alpha\mu\langle v_{k+1}-x^\star,x_{k+1}-v_{k+1}\rangle
			=
			\frac{\alpha\mu}{2}\left(\|x_{k+1}-x^\star\|^2-\|x_{k+1}-v_{k+1}\|^2-\|v_{k+1}-x^\star\|^2\right).
			$$
			The negative terms 
			$$
			-	D_{\mathcal E}(z_k^+,z_{k+1};\mu)
			=
			-	D_f(x_k^+,x_{k+1}) - \frac{\mu}{2}\|v_k-v_{k+1}\|^2.
			$$
			We will use those negative terms to cancel positive terms in \eqref{equ137} and \eqref{equ135} and obtain
			\begin{equation}\label{equ138}
				\begin{aligned}
					&\mathbb{E}\left[\mathcal{E}(z_{k+1};\mu)-\mathcal{E}(z_k^+;\mu)\mid\mathcal F_{k+1}\right] \\
					\le{}&
					-\alpha\mathbb{E}\left[\mathcal{E}(z_{k+1};\mu)\mid\mathcal F_{k+1}\right]
					+
					\frac{\tilde{\alpha}^2(1+\alpha)}{2\mu}
					\mathbb{E}\left[\|g(x_{k+1})\|^2\mid\mathcal F_{k+1}\right].
				\end{aligned}
			\end{equation}
			Finally, the conditional second-moment bound gives
			$$
			\mathbb{E}\left[\|g(x_{k+1})\|^2\mid\mathcal F_{k+1}\right]
			\le
			(1+\sigma^2)\|\nabla f(x_{k+1})\|^2.
			$$
			Substituting this into~\eqref{equ138} proves~\eqref{lemma7eq}.
		\end{proof}
		Substituting the bound from Lemma~\ref{lemma7}, together with \eqref{equ154_}, into the decomposition \eqref{equ154}, we obtain
		\begin{equation}\label{equ155}
			\begin{aligned}
				\mathbb{E}\left[\mathcal{E}(z_{k+1}^+; \mu) - \mathcal{E}(z_{k}^+; \mu)\right]  \le \mathbb{E}\left[-\alpha \mathcal{E}(z_{k+1}; \mu)  +\frac{\tilde{\alpha}^2\alpha (1+\sigma^2)}{2\mu} \| \nabla f(x_{k+1}) \|^2 \right]
			\end{aligned}
		\end{equation} 
		It remains to rewrite the term \(-\alpha\mathcal{E}(z_{k+1};\mu)\) in terms of
		\(\mathcal{E}(z_{k+1}^+;\mu)\). By the definition of \(\mathcal{E}\),
		\[
		\begin{aligned}
			\mathbb{E}\left[-\alpha \mathcal{E}(z_{k+1};\mu)\right]
			=
			\mathbb{E}\left[
			-\alpha\bigl(f(x_{k+1})-f(x^\star)\bigr)
			-\frac{\alpha\mu}{2}\|v_{k+1}-x^\star\|^2
			\right].
		\end{aligned}
		\]
		Applying the same descent estimate to the auxiliary step with step size
		\(\tilde{\alpha}\beta\) gives
		\[
		\mathbb{E}\left[-\alpha f(x_{k+1})\right]
		\le
		\mathbb{E}\left[
		-\alpha f(x_{k+1}^+)
		-\frac{\alpha\tilde{\alpha}\beta}{2}
		\|\nabla f(x_{k+1})\|^2
		\right].
		\]
		Since \(\tilde{\alpha}\beta=\frac{\tilde{\alpha}^2(1+\sigma^2)}{\mu}\), we obtain
		\[
		\begin{aligned}
			\mathbb{E}\left[-\alpha \mathcal{E}(z_{k+1};\mu)\right]
			&\le
			\mathbb{E}\left[
			-\alpha\bigl(f(x_{k+1}^+)-f(x^\star)\bigr)
			-\frac{\alpha\mu}{2}\|v_{k+1}-x^\star\|^2
			\right.\\
			&\qquad\left.
			-\frac{\alpha\tilde{\alpha}^2(1+\sigma^2)}{2\mu}
			\|\nabla f(x_{k+1})\|^2
			\right] \\
			&=
			\mathbb{E}\left[
			-\alpha\mathcal{E}(z_{k+1}^+;\mu)
			-\frac{\alpha\tilde{\alpha}^2(1+\sigma^2)}{2\mu}
			\|\nabla f(x_{k+1})\|^2
			\right].
		\end{aligned}
		\]
		Substituting this estimate into the previous inequality cancels the remaining positive gradient-norm term. Therefore,
		\[
		\mathbb{E}\left[
		\mathcal{E}(z_{k+1}^+;\mu)
		-
		\mathcal{E}(z_k^+;\mu)
		\right]
		\le
		\mathbb{E}\left[-\alpha\mathcal{E}(z_{k+1}^+;\mu)\right].
		\]
		Rearranging gives
		\[
		\mathbb{E}\left[\mathcal{E}(z_{k+1}^+;\mu)\right]
		\le
		\frac{1}{1+\alpha}
		\mathbb{E}\left[\mathcal{E}(z_k^+;\mu)\right].
		\]
		Since \(\alpha=\frac{\tilde{\alpha}}{1-\tilde{\alpha}}\), we have
		\(
		\frac{1}{1+\alpha}=1-\tilde{\alpha}.
		\)
		Hence,
		\[
		\mathbb{E}\left[\mathcal{E}(z_{k+1}^+;\mu)\right]
		\le
		(1-\tilde{\alpha})
		\mathbb{E}\left[\mathcal{E}(z_k^+;\mu)\right]
		\le
		(1-\tilde{\alpha})^{k+1}
		\mathbb{E}\left[\mathcal{E}(z_0^+;\mu)\right].
		\]
		This proves Theorem~\ref{theorem2}.
		
		\subsection{Proof of Corollaries~\ref{corollary1} and~\ref{corollary3}}\label{corollaryanalysis}
		
		We begin by recalling the definition of almost sure convergence.
		\begin{definition}
			Let $\{x_k^+\}$ be generated by an algorithm. We say that $f(x_k^+)$ converges almost surely to $f(x^\star)$ if, for every $\varepsilon>0$,
			\[
			\mathbb{P}\left(
			\left\{
			\left|f(x_k^+) - f(x^\star)\right| \ge \varepsilon
			\quad \mathrm{i.o.}
			\right\}
			\right)=0,
			\]
			where $\mathrm{i.o.}$ stands for infinitely often.
		\end{definition}
		
		Since $x^\star$ is a global minimizer, we have
		\(
		f(x_k^+) - f(x^\star) \ge 0.
		\)
		Thus, for any $\varepsilon>0$, Markov's inequality gives
		\[
		\mathbb{P}\left(f(x_k^+) - f(x^\star) \ge \varepsilon\right)
		\le
		\frac{1}{\varepsilon}
		\mathbb{E}\left[f(x_k^+) - f(x^\star)\right].
		\]
		
		In the strongly convex case, Theorems~\ref{theorem1} and~\ref{theorem2} imply that
		\(
		\mathbb{E}\left[f(x_k^+) - f(x^\star)\right]
		\le Cq^k
		\)
		for some constants $C>0$ and $q\in(0,1)$. Hence
		\[
		\sum_{k=1}^{\infty}
		\mathbb{P}\left(f(x_k^+) - f(x^\star) \ge \varepsilon\right)
		\le
		\frac{C}{\varepsilon}
		\sum_{k=1}^{\infty}q^k
		<\infty.
		\]
		
		In the convex case, Theorems~\ref{theorem1} and~\ref{theorem3} imply that
		\(
		\mathbb{E}\left[f(x_k^+) - f(x^\star)\right]
		\le \frac{C}{k^2}
		\)
		for some constant $C>0$. Hence
		\[
		\sum_{k=1}^{\infty}
		\mathbb{P}\left(f(x_k^+) - f(x^\star) \ge \varepsilon\right)
		\le
		\frac{C}{\varepsilon}
		\sum_{k=1}^{\infty}\frac{1}{k^2}
		<\infty.
		\]
		
		Therefore, in both cases,
		\(
		\sum_{k=1}^{\infty}
		\mathbb{P}\left(f(x_k^+) - f(x^\star) \ge \varepsilon\right)
		<\infty.
		\)
		It follows that
		\[
		\begin{aligned}
			\mathbb{P}\left(
			\left\{
			\left|f(x_k^+) - f(x^\star)\right| \ge \varepsilon
			\quad \mathrm{i.o.}
			\right\}
			\right) 
			&=
			\mathbb{P}\left(
			\bigcap_{N=1}^{\infty}
			\bigcup_{k\ge N}
			\left\{
			f(x_k^+) - f(x^\star)\ge \varepsilon
			\right\}
			\right) \\
			&=
			\lim_{N\to\infty}
			\mathbb{P}\left(
			\bigcup_{k\ge N}
			\left\{f(x_k^+) - f(x^\star)\ge \varepsilon
			\right\}
			\right) \\
			&\le
			\lim_{N\to\infty}
			\sum_{k\ge N}
			\mathbb{P}\left(
			f(x_k^+) - f(x^\star) \ge \varepsilon
			\right)
			=0.
		\end{aligned}
		\]
		Here the last equality follows from the convergence of the probability series above. Since $\varepsilon>0$ is arbitrary, we conclude that
		\[
		f(x_k^+) \to f(x^\star)
		\qquad \text{almost surely}.
		\]
		
		\section{Conclusion}\label{conclusion}
		We proposed SHANG and SHANG++, two stochastic accelerated methods designed for training under multiplicative noise scaling. SHANG++ augments SHANG with a simple damping correction, which improves robustness in noisy regimes while preserving fast convergence. Under the MNS condition, we established accelerated convergence guarantees in both convex and strongly convex settings, together with explicit parameter choices. Empirically, on convex benchmarks as well as image classification and reconstruction tasks, a single hyperparameter configuration performs consistently well and often maintains accuracy within one percentage point of the noise-free setting for \(\sigma\le0.5\) in our tests. Compared with other stochastic momentum methods, SHANG++ is more stable under small-batch or high-noise conditions and is competitive with widely used optimizers such as Adam. As natural extensions of the present framework toward adaptive optimization, in follow-up work we have developed Adam-HNAG \cite{yu2026adamhnagconvergentreformulationadam} and Adam-SHANG \cite{yu2026adamshangconvergentadamtypemethod}, which incorporate Adam-style adaptive preconditioning into the HNAG and SHANG frameworks, respectively. Finally, the strong empirical performance on nonconvex problems suggests that extending the analysis beyond convexity is an interesting direction for future work.

		\FloatBarrier
		\bibliographystyle{elsarticle-num}
		\bibliography{refs}

@misc{ChenLuo2019optimization,
	title={First order optimization methods based on Hessian-driven Nesterov accelerated gradient flow}, 
	author={Long Chen and Hao Luo},
	year={2019},
	eprint={1912.09276},
	archivePrefix={arXiv},
	primaryClass={math.OC},
}

@misc{GoujauTaylorDieule2025Provable,
	title={Provable non-accelerations of the heavy-ball method}, 
	author={Baptiste Goujaud and Adrien Taylor and Aymeric Dieuleveut},
	year={2025},
	eprint={2307.11291},
	archivePrefix={arXiv},
	primaryClass={math.OC},
}

@article{AssranRabbat2020,
	title={On the Convergence of Nesterov's Accelerated Gradient Method in Stochastic Settings},
	author={Mahmoud Assran and Michael G. Rabbat},
	journal={ArXiv},
	year={2020},
	volume={abs/2002.12414}
}

@article{AujolDossal2015,
	author = {Aujol, J.-F. and Dossal, Ch.},
	title = {Stability of Over-Relaxations for the Forward-Backward Algorithm, Application to FISTA},
	journal = {SIAM Journal on Optimization},
	volume = {25},
	number = {4},
	pages = {2408-2433},
	year = {2015},
	doi = {10.1137/140994964}
}

@article{Bollapragada2022,
	title={On the fast convergence of minibatch heavy ball momentum},
	volume={45},
	ISSN={1464-3642},
	DOI={10.1093/imanum/drae033},
	number={3},
	journal={IMA Journal of Numerical Analysis},
	author={Bollapragada, Raghu and Chen, Tyler and Ward, Rachel},
	year={2024},
	pages={1397–1424} }

@misc{ChenLuo2021,
	title={A Unified Convergence Analysis of First Order Convex Optimization Methods via Strong Lyapunov Functions}, 
	author={Long Chen and Hao Luo},
	year={2021},
	eprint={2108.00132},
	archivePrefix={arXiv},
	primaryClass={math.OC}
}

@misc{ChenXu2025,
	archiveprefix = {arXiv},
	author = {Long Chen and Zeyi Xu},
	eprint = {2510.16680},
	primaryclass = {math.OC},
	title = {HNAG++: A Super-Fast Accelerated Gradient Method for Strongly Convex Optimization},
	year = {2025}}

@article{DevolderGlineurNesterov2014,
	author = {Olivier Devolder and Fran{\c{c}}ois Glineur and Yurii Nesterov},
	journal = {Mathematical Programming},
	pages = {37--75},
	title = {First-Order Methods of Smooth Convex Optimization with Inexact Oracle},
	volume = {146},
	year = {2014}}

@inproceedings{even2021,
	author = {Even, Mathieu and Berthier, Rapha\"{e}l and Bach, Francis and Flammarion, Nicolas and Gaillard, Pierre and Hendrikx, Hadrien and Massouli\'{e}, Laurent and Taylor, Adrien},
	title = {A continuized view on nesterov acceleration for stochastic gradient descent and randomized gossip},
	year = {2021},
	isbn = {9781713845393},
	booktitle = {Proceedings of the 35th International Conference on Neural Information Processing Systems},
	numpages = {13}
}

@InProceedings{GaneshEtAl2023,
	title = 	 {Does Momentum Help in Stochastic Optimization? {A} Sample Complexity Analysis.},
	author =       {Ganesh, Swetha and Deb, Rohan and Thoppe, Gugan and Budhiraja, Amarjit},
	booktitle = 	 {Proceedings of the Thirty-Ninth Conference on Uncertainty in Artificial Intelligence},
	pages = 	 {602--612},
	year = 	 {2023},
	volume = 	 {216}
}

@inproceedings{GuptaSiegelWojtowytsch2024,
	author = {Gupta, Kanan and Siegel, Jonathan W. and Wojtowytsch, Stephan},
	title = {Nesterov acceleration despite very noisy gradients},
	year = {2024},
	isbn = {9798331314385},
	booktitle = {Proceedings of the 38th International Conference on Neural Information Processing Systems},
	articleno = {653},
	numpages = {51}
}

@book{HastieTibshiraniFriedman2009,
	author = {Trevor Hastie and Robert Tibshirani and Jerome Friedman},
	publisher = {Springer},
	title = {The Elements of Statistical Learning: Data Mining, Inference, and Prediction},
	year = {2009}}

@INPROCEEDINGS{HeZhangRenSun2016,
	author={He, Kaiming and Zhang, Xiangyu and Ren, Shaoqing and Sun, Jian},
	booktitle={2016 IEEE Conference on Computer Vision and Pattern Recognition (CVPR)}, 
	title={Deep Residual Learning for Image Recognition}, 
	year={2016},
	pages={770-778},
	doi={10.1109/CVPR.2016.90}}

@misc{HermantEtAl2025,
	title={Gradient correlation is a key ingredient to accelerate SGD with momentum}, 
	author={Julien Hermant and Marien Renaud and Jean-François Aujol and Charles Dossal and Aude Rondepierre},
	year={2025},
	eprint={2410.07870},
	archivePrefix={arXiv}
}

@misc{jain2018,
	title={Accelerating Stochastic Gradient Descent For Least Squares Regression}, 
	author={Prateek Jain and Sham M. Kakade and Rahul Kidambi and Praneeth Netrapalli and Aaron Sidford},
	year={2018},
	eprint={1704.08227},
	archivePrefix={arXiv}
}

@misc{KidambiEtAl2018,
	title={On the insufficiency of existing momentum schemes for Stochastic Optimization}, 
	author={Rahul Kidambi and Praneeth Netrapalli and Prateek Jain and Sham M. Kakade},
	year={2018},
	eprint={1803.05591},
	archivePrefix={arXiv}

}

@inproceedings{KingmaBa2015,
	author       = {Diederik P. Kingma and
	Jimmy Ba},
	title        = {Adam: {A} Method for Stochastic Optimization},
	booktitle    = {3rd International Conference on Learning Representations, {ICLR} 2015},
	year         = {2015},
}

@book{Klenke2013,
	author = {Achim Klenke},
	publisher = {Springer},
	title = {Probability Theory: A Comprehensive Course},
	year = {2013}}

@inproceedings{Krizhevsky2009,
	title={Learning Multiple Layers of Features from Tiny Images},
	author={Alex Krizhevsky},
	year={2009},
	url={https://api.semanticscholar.org/CorpusID:18268744}
}

@InProceedings{laborde2020,
	title = 	 {A Lyapunov analysis for accelerated gradient methods: from deterministic to stochastic case},
	author =       {Laborde, Maxime and Oberman, Adam},
	booktitle = 	 {Proceedings of the Twenty Third International Conference on Artificial Intelligence and Statistics},
	pages = 	 {602--612},
	year = 	 {2020},
	volume = 	 {108}
}

@misc{Hodgkinson2021,
	title={Multiplicative noise and heavy tails in stochastic optimization}, 
	author={Liam Hodgkinson and Michael W. Mahoney},
	year={2020},
	eprint={2006.06293},
	archivePrefix={arXiv}
}

@article{LeCunBottouBengioHaffner1998,
	author = {Yann LeCun and L{\'e}on Bottou and Yoshua Bengio and Patrick Haffner},
	journal = {Proceedings of the IEEE},
	number = {11},
	pages = {2278--2324},
	title = {Gradient-Based Learning Applied to Document Recognition},
	volume = {86},
	year = {1998}}

@article{LessardRechtPackard2016,
	author = {Lessard, Laurent and Recht, Benjamin and Packard, Andrew},
	title = {Analysis and Design of Optimization Algorithms via Integral Quadratic Constraints},
	journal = {SIAM Journal on Optimization},
	volume = {26},
	number = {1},
	pages = {57-95},
	year = {2016},
	doi = {10.1137/15M1009597}
}

@article{LiuEtAl2018,
	title={Toward Deeper Understanding of Nonconvex Stochastic Optimization with Momentum using Diffusion Approximations},
	author={Tianyi Liu and Zhehui Chen and Enlu Zhou and Tuo Zhao},
	journal={ArXiv},
	year={2018},
	volume={abs/1802.05155}
}

@misc{LiuBelkin2020,
	title={Accelerating SGD with momentum for over-parameterized learning}, 
	author={Chaoyue Liu and Mikhail Belkin},
	year={2019},
	eprint={1810.13395},
	archivePrefix={arXiv}
}

@article{Nesterov1983,
	title = {A method for solving the convex programming problem with convergence rate {$O(1/k^2)$}},
	author={Yurii Nesterov},
	journal={Proceedings of the USSR Academy of Sciences},
	year={1983},
	volume={269},
	pages={543-547},
}

@article{Nesterov2013,
	author = {Nesterov, Yu.},
	title = {Efficiency of Coordinate Descent Methods on Huge-Scale Optimization Problems},
	journal = {SIAM Journal on Optimization},
	volume = {22},
	number = {2},
	pages = {341-362},
	year = {2012},
	doi = {10.1137/100802001},
}

@article{Polyak1964,
	title = {Some methods of speeding up the convergence of iteration methods},
	journal = {USSR Computational Mathematics and Mathematical Physics},
	volume = {4},
	number = {5},
	pages = {1-17},
	year = {1964},
	issn = {0041-5553},
	doi = {https://doi.org/10.1016/0041-5553(64)90137-5},
	author = {B.T. Polyak},
}

@misc{RonnebergerFischerBrox2015,
	title={U-Net: Convolutional Networks for Biomedical Image Segmentation}, 
	author={Olaf Ronneberger and Philipp Fischer and Thomas Brox},
	year={2015},
	eprint={1505.04597},
	archivePrefix={arXiv},
}

@misc{thulasidasan2020mixuptrainingimprovedcalibration,
	archiveprefix = {arXiv},
	author = {Sunil Thulasidasan and Gopinath Chennupati and Jeff Bilmes and Tanmoy Bhattacharya and Sarah Michalak},
	eprint = {1905.11001},
	title = {On Mixup Training: Improved Calibration and Predictive Uncertainty for Deep Neural Networks},
	year = {2020}}

@inproceedings{VaswaniBachSchmidt2019,
	author = {Sharan Vaswani and Francis Bach and Mark Schmidt},
	booktitle = {International Conference on Artificial Intelligence and Statistics (AISTATS)},
	pages = {1195--1204},
	title = {Fast and Faster Convergence of {SGD} for Over-Parameterized Models and an Accelerated Perceptron},
	year = {2019}}

@misc{WuDuWard2019,
	title={Global Convergence of Adaptive Gradient Methods for An Over-parameterized Neural Network}, 
	author={Xiaoxia Wu and Simon S. Du and Rachel Ward},
	year={2019},
	eprint={1902.07111},
	archivePrefix={arXiv},
}

@inproceedings{WuWangSu2022a,
	title={The alignment property of SGD noise and how it helps select flat minima: A stability analysis},
	author={Lei Wu and Mingze Wang and Weijie Su},
	booktitle={Neural Information Processing Systems},
	year={2022},
}

@misc{yu2026adamhnagconvergentreformulationadam,
	title={Deterministic Adam-Inspired Methods with Accelerated Convergence Rate}, 
	author={Yaxin Yu and Long Chen and Zeyi Xu},
	year={2026},
	eprint={2604.08742},
	archivePrefix={arXiv},
	primaryClass={math.OC},
	url={https://arxiv.org/abs/2604.08742}, 
}

@misc{yu2026adamshangconvergentadamtypemethod,
	title={Adam-SHANG: A Convergent Adam-Type Method for Stochastic Smooth Convex Optimization}, 
	author={Yaxin Yu and Long Chen and Minfu Feng},
	year={2026},
	eprint={2605.12878},
	archivePrefix={arXiv},
	primaryClass={math.OC},
	url={https://arxiv.org/abs/2605.12878}, 
}
	\end{document}